\documentclass[11pt]{article}

\usepackage{amsmath,amsthm,amssymb,amsfonts, graphicx}
\usepackage{graphics}
\usepackage{authblk}
\usepackage{upgreek}
\usepackage{amsrefs,hyperref}
\theoremstyle{plain}
\newtheorem{theorem}{Theorem}[section]

\newtheorem{corollary}[theorem]{Corollary}
\newtheorem{definition}[theorem]{Definition}

\newtheorem{proposition}[theorem]{Proposition}
\newtheorem{conjecture}[theorem]{Conjecture}
\theoremstyle{definition}
\newtheorem{remark}[theorem]{Remark}
\numberwithin{equation}{section}
\newcommand{\diff}{\mathop{}\!\mathrm{d}}
\DeclareMathOperator*{\Hess}{Hess}

\DeclareMathOperator*{\esssup}{ess\,sup}

\title{Navier--Stokes regularity criteria in sum spaces}
\author[1]{Evan Miller}
\affil[1]{McMaster University, Department of Mathematics,

millee14@mcmaster.ca}

\begin{document}

\maketitle

\begin{abstract}
In this paper, we will consider regularity criteria for the Navier--Stokes equation in mixed Lebesgue sum spaces.
In particular, we will prove regularity criteria that only require control of the velocity, vorticity, or the positive part of the second eigenvalue of the strain matrix, in the sum space of two scale critical spaces.
This represents a significant step forward, because each sum space regularity criterion covers a whole family of scale critical regularity criteria in a single estimate.
In order to show this, we will also prove a new inclusion and inequality for sum spaces in families of mixed Lebesgue spaces with a scale invariance that is also of independent interest.

\end{abstract}

\section{Introduction}
The Navier--Stokes equation is one of the fundamental equations of fluid dynamics. For incompressible flow, meaning the density of the fluid is constant, the Navier--Stokes equation is given by
\begin{align} \label{Navier}
    \partial_t u -\nu \Delta u +(u \cdot \nabla )u +\nabla p
    &= 0\\
    \nabla \cdot u&=0,
\end{align}
where $u$ is the velocity, $\nu>0$ is the kinematic viscosity, and $p$ is the pressure. When $\nu=0,$ this reduces to the Euler equation for inviscid fluids. The first equation expresses Newton's second law, with 
$\partial_t u+ (u\cdot \nabla)u$ the acceleration in the Lagrangian frame, and 
$\nu \Delta u+\nabla p,$ the force divided by the mass---noting that the pressure $p$ in \eqref{Navier} is actually the physical pressure divided by the density of the fluid in question.
We will note here that the pressure can be determined entirely in terms of the velocity using the divergence free constraint and inverting the Laplacian with
\begin{equation}
    -\Delta p=\sum_{i,j=1}^3 
    \frac{\partial u_j}{\partial x_i}
    \frac{\partial u_i}{\partial x_j}.
\end{equation}
Because the gradient of pressure is a Lagrange multiplier for the divergence free constraint, it is possible to drop this term by using the Helmholtz projection onto the space of divergence free vector fields. We can rewrite the equation as
\begin{equation}
    \partial_t u -\nu \Delta u +P_{df}((u\cdot \nabla)u)=0.
\end{equation}

Two other crucially important objects for the study of the Navier--Stokes equation are the strain and the vorticity. The strain is the symmetric part of $\nabla u,$ with 
\begin{equation}
    S_{ij}=\frac{1}{2}\left(
    \partial_i u_j+ \partial_j u_i\right).
\end{equation}
The evolution equation for the strain is given by
\begin{equation}
    \partial_t S -\nu\Delta S +(u\cdot \nabla )S 
    +S^2 +\frac{1}{4}\omega\otimes\omega
    -\frac{1}{4}|\omega|^2 I_3+\Hess(p)
    =0.
\end{equation}
The vorticity is given by $\omega=\nabla \times u,$ and is a vector representation of the anti-symmetric part of $\nabla u.$ The evolution equation for the vorticity is given by
\begin{equation}
    \partial_t \omega- \nu\Delta\omega +(u\cdot \nabla)\omega
    -S\omega=0.
\end{equation}
While the velocity tells us how a parcel of the fluid is advected, the vorticity tells us how a parcel of the fluid is rotated, and the strain tells us how a parcel of the fluid is deformed, and for that reason is also known as the deformation matrix.

In his ground breaking work on the Navier--Stokes equation \cite{Leray}, Leray proved the existence of weak solutions to the Navier--Stokes equation in $L^\infty\left([0,+\infty); L^2 \right) 
\cap L^2 \left([0,+\infty);\dot{H}^1 \right)$ for generic initial data $u^0 \in L^2.$ Leray proved the existence of weak solutions in the sense of distributions satisfying the energy inequality
\begin{equation}
    \frac{1}{2}\|u(\cdot,t)\|_{L^2}^2+
    \nu\int_0^t \|\nabla u(\cdot,\tau)\|_{L^2}^2 \diff\tau
    \leq
    \frac{1}{2}\left\|u^0\right\|_{L^2}^2,
\end{equation}
for all $t>0.$ For strong solutions, this inequality holds with equality.
While Leray-Hopf weak solutions, as such solutions are generally known, must exist globally in time, they are not known to be either smooth or unique.

Kato and Fujita developed the notion of mild solutions based on the heat semi-group \cite{KatoFujita}. Mild solutions of the Navier--Stokes equation satisfy the equation
\begin{equation}
    \partial_t u -\nu\Delta u=-P_{df}((u\cdot\nabla)u),
\end{equation}
in the sense of convolution with the heat kernel 
as in Duhamel's formula.

Kato and Fujita proved that mild solutions must exist locally in time for initial data in $\dot{H}^1$ uniformly in terms of the $\dot{H}^1$ norm, and furthermore that such solutions must be unique and have higher regularity \cite{KatoFujita}. 
We will give a precise statement of the definition of a mild solution and Kato and Fujita's local existence theorem in section \ref{Defintions}.

While, unlike Leray-Hopf weak solutions, mild solutions must be smooth and unique, they may not exist globally in time. This represents a major conundrum, because while it is not really a problem if smooth solutions of the Navier--Stokes equation develop singularities in finite-times---mathematical singularities describe many phenomena that actually exist in nature, from the shock waves that develop when the sound barrier is broken to the formation of black holes---for any notion of a solution to be physically meaningful, there should at least be a guarantee that solutions are unique, as the Navier--Stokes equation is a deterministic model.

The Navier--Stokes equation has a scale invariance. If $u$ is a solution of the Navier--Stokes equation then so is $u^\lambda,$ for all $\lambda>0,$ where 
\begin{equation} \label{Scaling}
    u^\lambda (x,t) =\lambda u(\lambda x,\lambda^2 t).
\end{equation}
We will note that for initial data, this rescaling becomes
\begin{equation} \label{ScalingInitial}
    u^{0,\lambda} (x) =\lambda u^0 (\lambda x).
\end{equation}
Kato proved the existence of smooth solutions globally in time for small initial data in $L^3,$ which is critical with respect to this scaling in \cite{KatoL3}, and this was later extending by Koch and Tataru to the critical space $BMO^{-1}$ in \cite{KochTataru}.

It is also possible to guarantee that a solution must remain smooth as long as there is control on some scale critical quantity is controlled. The Ladyzhenskaya-Prodi-Serrin regularity criterion \cites{Ladyzhenskaya,Prodi,Serrin} states that if $u\in L^p_T L^q_x, \frac{2}{p}+\frac{3}{q}=1, 3<q\leq+\infty,$ then the solution is smooth and can be continued to a smooth solution for some time $\Tilde{T}>T.$
In particular, if $T_{max}<+\infty,$ then
\begin{equation}
    \int_0^{T_{max}}\|u(\cdot,t)\|_{L^q}^p \diff t=+\infty,
\end{equation}
where $T_{max}$ is the maximal time of existence for a mild solution for some initial data $u^0\in H^1_{df}.$
This was then extended to the endpoint case $p=+\infty,q=3$ by Escauriaza, Seregin, and \v{S}ver\'ak in \cite{ESS}, where they showed that if $T_{max}<+\infty,$ then
\begin{equation} \label{L3ESS}
    \limsup_{t \to T_{max}}\|u(\cdot,t)\|_{L^3}
    =+\infty.
\end{equation}
This result was improved by Seregin in \cite{SereginL3}, where the the limit supremum in \eqref{L3ESS} was replaced with the limit, and recently further improved by Tao, who proved a triply logarithmic lower bound on the rate of blowup of the $L^3$ norm \cite{TaoL3}. Somewhat more precisely, Tao showed that for an absolute constant $c>0,$ if $T_{max}<+\infty$,
\begin{equation}
    \limsup_{t \to T_{max}}
    \frac{\|u(\cdot,t)\|_{L^3}}
    {\left(\log\log\log \frac{1}{T_{max}-t}\right)^c}
    =+\infty.
\end{equation}
The Ladyzhenskaya-Prodi-Serrin regularity criterion has also been strengthened to require control in a family of scale critical spaces involving the endpoint Besov space,
$L^p_T \dot{B}^\sigma_{\infty,\infty}$
\cites{ChenZhangBesov,KOTbesov,KozonoShimadaBesov},
while the Escauriaza-Seregin-\v{S}ver\'ak regularity criterion has been strengthened to require control in time of the scale-critical nonendpoint Besov spaces, $L^\infty_T \dot{B}^{-1+\frac{3}{p}}_{p,q}$ 
\cites{AlbrittonBesov,GKPbesov}.

In this paper, we will extend the Ladyzhenskaya-Prodi-Serrin regularity criterion to the sum space $L^p_T L^q_x+ L^2_T L^\infty_x,$ for all $\frac{2}{p}+\frac{3}{q}=1, 3<q<+\infty.$
Our precise result is as follows.

\begin{theorem} \label{VelocitySumSpaceIntro}
Suppose $u\in C\left(\left[0,T_{max}\right);
\dot{H}^1_{df}\right)$ is a smooth solution of the Navier--Stokes equation. Let $3<q<+\infty,
\frac{2}{p}+\frac{3}{q}=1,$ and let 
$ u=v+\sigma.$ Then for all $0<T<T_{max}$
\begin{equation}
    \|\nabla u(\cdot,T)\|_{L^2}^2 \leq 
    \left\| \nabla u^0 \right\|_{L^2}^2
    \exp\left(\frac{C_p}{\nu^{p-1}} \int_0^T 
    \|v(\cdot,t)\|_{L^q}^p \diff t
    +\frac{1}{\nu} \int_0^T \|\sigma(\cdot,t)\|_{L^\infty}^2
    \diff t \right),
\end{equation}
where $C_p$ depends only on $p.$
In particular if $T_{max}<+\infty,$ then
\begin{equation}
    \frac{C_p}{\nu^{p-1}} \int_0^{T_{max}} 
    \|v(\cdot,t)\|_{L^q}^p \diff t
    +\frac{1}{\nu} \int_0^{T_{max}} 
    \|\sigma(\cdot,t)\|_{L^\infty}^2 \diff t
    =+\infty.
\end{equation}
\end{theorem}

We will note that this is a significant advance because the regularity criterion in the sum space 
$L^p_T L^q_x+ L^2_T L^\infty_x$ 
contains within it the whole family of regularity criteria
in the spaces $L^{p'}_T L^{q'}_x,$ where $\frac{2}{p'}+\frac{3}{q'}=1,$ and $q\leq q' \leq +\infty.$
We will prove this inclusion in section \ref{SumSpaceSection}.

There is a very large literature on regularity criteria for the Navier--Stokes equation. In addition to the aforementioned 
Ladyzhenskaya-Prodi-Serrin regularity criterion, regularity criteria have also been proven involving the vorticity, particularly the celebrated Beale-Kato-Madja regularity criterion \cite{BKM}, which applies both to solutions of the Euler and Navier--Stokes equations. There have also been a number of scale-critical, component-reduction-type regularity criteria that only require control on a certain part of the solution, including just two components of the vorticity $(\omega_1,\omega_2,0)$---or equivalently $e_3 \times \omega$ \cite{ChaeVort}, 
the derivative in just one direction
$\partial_3 u$ \cite{KukavicaZiane}, 
and just one component of the velocity $u_3$ \cites{CheminZhang,CheminZhangZhang}.

All of these component reduction results have a physical significance in that they can be seen as saying that blowup must in some sense be fully three dimensional and isotropic. For the 2D Navier--Stokes equation, there are global smooth solutions, and $e_3 \times \omega, \partial_3 u,$ and $u_3$ are all identically zero. If solutions of the 3D Navier--Stokes equation is treated as a perturbation of the 2D Navier--Stokes equation, this means that these regularity criteria can be seen as perturbation conditions.
If $e_3 \times \omega, \partial_3 u,$ or $u_3$ remain finite in the appropriate scale critical space, then our solution is close enough to being 2D to guarantee regularity. This is also consistent with the phenomenological picture of turbulence developed by Kolmogorov, which rests on the assumption that turbulence is locally isotropic at sufficiently small scales \cite{Kolmogorov}.

Another component reduction regularity criterion involves the positive part of the second eigenvalue of the strain matrix. If $T_{max}<+\infty,$ then for all $\frac{2}{p}+\frac{3}{q}=2, \frac{3}{2}<q\leq +\infty,$
\begin{equation} \label{EigenCrit}
    \int_0^{T_{max}}\|\lambda_2^+(\cdot,t)\|_{L^q}^p
    \diff t=+\infty.
\end{equation}
This was first proven by Neustupa and Penel \cites{NeustupaPenel1,NeustupaPenel2,NeustupaPenel3} and independently by the author using somewhat different methods in \cite{MillerStrain}.
This component reduction regularity criterion has a particular geometric interpretation. When the strain has two positive eigenvalues, it means that there is stretching in two directions and compression more strongly in a third, and therefore there is planar stretching and axial compression. When the strain has two negative eigenvalues, it means that there is compression in two directions, and stretching more strongly in a third, and therefore this is planar compression and axial stretching. The regularity criterion in \eqref{EigenCrit} therefore implies that finite-time blowup for the Navier--Stokes equation requires unbounded planar stretching.

We can generalize the regularity criterion \eqref{EigenCrit} to a regularity criterion on 
$\lambda_2^+$ in the sum space $L^p_T L^q_x+ L^1_T L^\infty_x$.

\begin{theorem} \label{StrainSumSpaceIntro}
Suppose $u\in C\left(\left[0,T_{max}\right);
\dot{H}^1_{df}\right)$ is a smooth solution of the Navier--Stokes equation. 
Let $\lambda_1(x,t) \leq \lambda_2(x,t) \leq \lambda_3(x,t)$ 
be the eigenvalues of the strain, $S(x,t)$,
and let $\lambda_2^+=\max\left(0,\lambda_2\right)$.
Let $\frac{3}{2}<q<+\infty, \frac{2}{p}+\frac{3}{q}=2,$ 
and let 
$\lambda_2^+ =f+g.$ 
Then for all $0<T<T_{max}$
\begin{equation}
    \|S(\cdot,T)\|_{L^2}^2 \leq 
    \left\| S^0 \right\|_{L^2}^2
    \exp\left(\frac{C_p}{\nu^{p-1}} \int_0^T 
    \|f(\cdot,t)\|_{L^q}^p \diff t
    +2\int_0^T \|g(\cdot,t)\|_{L^\infty} \diff t
    \right),
\end{equation}
where $C_p$ depends only on $p.$
In particular if $T_{max}<+\infty,$ then
\begin{equation}
    \frac{C_p}{\nu^{p-1}} \int_0^{T_{max}} 
    \|f(\cdot,t)\|_{L^q}^p \diff t
    +2\int_0^{T_{max}} \|g(\cdot,t)\|_{L^\infty} \diff t
    =+\infty.
\end{equation}
\end{theorem}

Finally, we will note that we can also express the regularity criterion for the vorticity in terms of sum space 
$L^p_T L^q_x+ L^1_T L^\infty_x$.

\begin{theorem} \label{VortSumSpaceIntro}
Suppose $u\in C\left(\left[0,T_{max}\right);
\dot{H}^1_{df}\right)$ is a smooth solution of the Navier--Stokes equation. Let $\frac{3}{2}<q<+\infty,
\frac{2}{p}+\frac{3}{q}=2,$ and let 
$\omega =v+\sigma.$ Then for all $0<T<T_{max}$
\begin{equation}
    \|\omega(\cdot,T)\|_{L^2}^2 \leq 
    \left\| \omega^0 \right\|_{L^2}^2
    \exp\left(\frac{C_p}{\nu^{p-1}} \int_0^T 
    \|v(\cdot,t)\|_{L^q}^p \diff t
    +\sqrt{2}\int_0^T \|\sigma(\cdot,t)\|_{L^\infty} \diff t
    \right),
\end{equation}
where $C_p$ depends only on $p.$
In particular if $T_{max}<+\infty,$ then
\begin{equation}
    \frac{C_p}{\nu^{p-1}} \int_0^{T_{max}} 
    \|v(\cdot,t)\|_{L^q}^p \diff t
    +\sqrt{2}\int_0^{T_{max}} \|\sigma(\cdot,t)\|_{L^\infty} \diff t
    =+\infty.
\end{equation}
\end{theorem}

Given that we have proven a number of regularity criteria in Lebesgue sum spaces of the form 
$L^p_t L^q_x + L^k_T L^\infty_x$,
it is a natural to consider the structure of such sum spaces in more detail. In particular, we will consider what spaces are contained in this sum space. If we are working with standard Lebesgue spaces, rather than mixed spaces, it is a well known result---see for instance Chapter 6.1 in \cite{Folland}---that for all $1\leq q' <q <+\infty,$
\begin{equation} \label{StandardInclusionIntroA}
    L^q \subset L^{q'}+L^\infty.
\end{equation}
In fact, it is straightforward to show this inclusion also holds if $L^q$ is replaced with $L^{q,\infty},$ the endpoint Lorentz space also known as weak $L^q,$ in which case we have
\begin{equation} \label{StandardInclusionIntro}
        L^{q,\infty} \subset L^{q'}+L^\infty.
\end{equation}

Note that in Theorem \ref{VelocitySumSpaceIntro} we have a scaling relation $\frac{2}{p}+\frac{3}{q}=1$, so if we take 
$(p',q'), (p,q),$ and $(2,\infty),$ satisfying this scaling relations and $q'<q<+\infty,$ then the point $(p,q)$ is in some sense in between $(p',q')$ and $(2,\infty)$, so there is a reason to expect we may have an inclusion of the form 
$L_T^p L_x^q \subset 
L_T^{p'} L_x^{q'}+L_T^2 L_x^\infty$
analogous to the inclusion
\eqref{StandardInclusionIntroA}.
We will show that this inclusion does hold, and we do in fact have a slightly stronger inclusion.

\begin{theorem} \label{MixedLpSumIntro}
Suppose $1\leq k<+\infty, 1\leq  m<+\infty,$ and suppose
\begin{equation}
    \frac{k}{p}+\frac{m}{q}=1,
\end{equation}
and
\begin{equation}
    \frac{k}{p'}+\frac{m}{q'}=1,
\end{equation}
with $m <q' < q< +\infty.$
Then 
\begin{equation}
    L_T^p L_x^{q,\infty} \subset L_T^{p'} L_x^{q'}
    +L_T^k L_x^\infty.
\end{equation}
In particular,
for all $f\in  L_T^p L_x^{q,\infty}$,
we have the explicit decomposition,
$f=g+h$ with 
$g \in L_T^{p'} L_x^{q'}, h \in L_T^k L_x^\infty,$ where
\begin{equation}
    g(x,t)= \begin{cases}
        f(x,t), &\text{if } |f(x,t)|>
        \|f(\cdot,t)\|_{L^{q,\infty}}^{\frac{p}{k}}
        \\
        0, &\text{if } |f(x,t)|\leq
        \|f(\cdot,t)\|_{L^{q,\infty}}^{\frac{p}{k}}
            \end{cases},
\end{equation}
and
\begin{equation}
    h(x,t)= \begin{cases}
        f(x,t), &\text{if } |f(x,t)| \leq
        \|f(\cdot,t)\|_{L^{q,\infty}}^{\frac{p}{k}}
        \\
        0, &\text{if } |f(x,t)| >
        \|f(\cdot,t)\|_{L^{q,\infty}}^{\frac{p}{k}}
            \end{cases},
\end{equation}
and we have the bounds
\begin{equation} 
    \int_0^T \|g(\cdot,t)\|_{L^{q'}}^{p'} \diff t
    \leq \left(\frac{q}{q-q'}\right)^\frac{p'}{q'} 
    \int_0^T \|f(\cdot,t)\|_{L^{q,\infty}}^p \diff t,
\end{equation}
and
\begin{equation} 
    \int_0^T \|h(\cdot,t)\|_{L^\infty}^k \diff t
    \leq \int_0^T 
    \|f(\cdot,t)\|_{L^{q,\infty}}^p \diff t.
\end{equation}
\end{theorem}

\begin{remark}
We know from Theorem \ref{MixedLpSumIntro} that for all
\begin{equation}
    \frac{2}{p}+\frac{3}{q}=1,
\end{equation}
and
\begin{equation}
    \frac{2}{p'}+\frac{3}{q'}=1,
\end{equation}
$3<q'<q<+\infty,$
\begin{equation}
    L_T^p L_x^{q,\infty} \subset L_T^{p'} L_x^{q'}
    +L_T^2 L_x^\infty.
\end{equation}
This means that the regularity criterion in Theorem \ref{VelocitySumSpaceIntro} in the sum space 
$L^{p'}_T L^{q'}_x+ L^2_T L^\infty_x$ 
contains within it the whole family of regularity criteria
in the spaces $L^{p}_T L^{q,\infty}_x,$ where $\frac{2}{p}+\frac{3}{q}=1,$ and $q'< q < +\infty.$
Theorem \ref{MixedLpSumIntro} implies that this whole family of scale critical regularity criteria can be contained in a single estimate.

Likewise we can see from Theorem \ref{MixedLpSumIntro} 
that for all
\begin{equation}
    \frac{1}{p}+\frac{3}{2q}=1,
\end{equation}
and
\begin{equation}
    \frac{1}{p'}+\frac{3}{2q'}=1,
\end{equation}
$\frac{3}{2}<q'<q<+\infty,$
\begin{equation}
    L_T^p L_x^{q,\infty} \subset L_T^{p'} L_x^{q'}
    +L_T^1 L_x^\infty.
\end{equation}
This implies that the regularity criteria in Theorems \ref{StrainSumSpaceIntro} and \ref{VortSumSpaceIntro}, 
on $\lambda_2^+$ and $\omega$ respectively, 
in the sum space 
$L^{p'}_T L^{q'}_x+ L^1_T L^\infty_x$ 
contain within themselves the whole family of regularity criteria
in the spaces $L^{p}_T L^{q,\infty}_x,$ where $\frac{2}{p}+\frac{3}{q}=2,$ and $q' < q < +\infty.$

These regularity criteria represent a significant advance not only because they have been improved to only requiring control on two different pieces of $u,\lambda_2^+$, or $\omega$ in two different scale critical spaces rather than requiring control on all of $u,\lambda_2^+$, or $\omega$ in a single scale critical space, but particularly because these regularity criteria contain a whole family of scale critical regularity criteria in a single estimate.
\end{remark}

\begin{remark}
For a large number of evolution equations in nonlinear PDEs, scaling laws and scale invariant spaces play a very important role. This is true not just for the Navier--Stokes equation, but also for the nonlinear Schr\"odinger equation, the nonlinear wave equation, and many other nonlinear evolution equations. Suppose we have a nonlinear evolution equation on $\mathbb{R}^d \times \mathbb{R}^+$ with a scale invariance
\begin{equation}
    u^\lambda(x,t)=\lambda^a u(\lambda x,\lambda^b t),
\end{equation}
with $0<a\leq n$ and $a \leq b$.
Then for all $\frac{b}{p}+\frac{n}{q}=a,$
the space $L_T^p L_x^q$ is scale invariant.

Applying Theorem \ref{MixedLpSumIntro} we can see that for three sets of exponents this family of scale invariant spaces
$(p',q'),(p,q)$ and $\left(\frac{b}{a},\infty\right),$ with $\frac{n}{a}<q'<q<+\infty,$
we have
\begin{equation}
    L_T^p L_x^{q,\infty} \subset L_T^{p'} L_x^{q'}
    +L_T^{\frac{b}{a}} L_x^\infty.
\end{equation}
This means that Theorem \ref{MixedLpSumIntro} may have broader applications to nonlinear evolution equations, as the growth of families of scale critical $L^p_T L^q_x$ norms is ubiquitous in the theory of nonlinear evolution equations.
\end{remark}

\begin{remark}
One natural question to consider about Theorem \ref{MixedLpSumIntro} is whether the inclusion will still hold if the control if relaxed to being in weak $L^p$ in time, in addition to weak $L^q$ in space. 
As we have already noted, for purely spatial variables we have 
\begin{equation}
        L^{q,\infty} \left(\mathbb{R}^3\right) \subset L^{q'}\left(\mathbb{R}^3\right)
        +L^\infty\left(\mathbb{R}^3\right),
\end{equation}
when $q'<q<+\infty$,
so this certainly gives some hope that the inclusion in Theorem \ref{MixedLpSumIntro} will still hold when the control in time is slightly relaxed, yielding
\begin{equation}
    L_T^{p,\infty} L_x^{q,\infty} \subset L_T^{p'} L_x^{q'}
    +L_T^k L_x^\infty.
\end{equation}
Based on the proof of Theorem \ref{MixedLpSumIntro}, it does not appear that this inclusion holds, however we do not have a counterexample at this time. We will discuss this more in section \ref{SumSpaceSection}, after we have proven Theorem \ref{MixedLpSumIntro}.
\end{remark}

\begin{conjecture} \label{MixedLpConjecture}
Suppose $1\leq k<+\infty, 1<m<+\infty,$ and suppose
\begin{equation}
    \frac{k}{p}+\frac{m}{q}=1,
\end{equation}
and
\begin{equation}
    \frac{k}{p'}+\frac{m}{q'}=1,
\end{equation}
with $m\leq q' < q< +\infty.$
Then
\begin{equation}
    L_T^{p,\infty} L_x^{q,\infty} \not\subset 
    L_T^{p'} L_x^{q'}+L_T^k L_x^\infty.
\end{equation}
\end{conjecture}

Finally we will consider the endpoint of the scale critical regularity criteria, where $p=+\infty.$ We will note that the requirement in Theorem \ref{VelocitySumSpaceIntro} that $q>3$ and the requirement in Theorems \ref{StrainSumSpaceIntro} and \ref{VortSumSpaceIntro} that $q>\frac{3}{2},$
imply that $p<+\infty$.
As we have previously mentioned, Escauriaza, Seregin, and \v{S}ver\'ak showed that 
if $T_{max}<+\infty,$ then
\begin{equation}
    \limsup_{t\to T_{max}}\|u(\cdot,t)\|_{L^3}=+\infty,
\end{equation}
which covers the endpoint case $L^\infty_T L^3_x.$ 
Applying the Sobolev inequality, this immediately implies that
\begin{equation}
    \limsup_{t\to T_{max}}
    \|\omega(\cdot,t)\|_{L^\frac{3}{2}}
    =+\infty,
\end{equation}
so we also have the endpoint regularity criterion for vorticity in $L^\infty_T L^\frac{3}{2}_x.$
The proof of the endpoint regularity criteria in \cite{ESS} is quite technical, and is not based on applying a Gr\"onwall estimate to control enstrophy growth, so in order to establish a regularity criteria for the velocity in 
$L^\infty_T L^3_x +L^2_T L^\infty_x$
or a regularity criteria for the vorticity in
$L^\infty_T L^\frac{3}{2}_x +L^1_T L^\infty_x,$
would require methods well beyond those used in this paper.
For the positive part of the second eigenvalue of the strain matrix, it still remains an open question whether $T_{max}<+\infty$ implies that
\begin{equation}
    \limsup_{t\to T_{max}}\|\lambda_2^+(\cdot,t)\|_{L^\frac{3}{2}}
    =+\infty.
\end{equation}
The author showed in \cite{MillerStrain} that if $T_{max}<+\infty,$ then
\begin{equation}
    \limsup_{t \to T_{max}}
    \left\|\lambda_2^+(\cdot,t)\right\|_{L^\frac{3}{2}}
    \geq 3\left(\frac{\pi}{2}\right)^\frac{4}{3} \nu.
\end{equation}
We are able to generalize this result to the sum space case.

\begin{theorem} \label{EigenEndpointIntro}
Suppose $u\in C\left(\left[0,T_{max}\right);
\dot{H}^1_{df}\right)$ is a smooth solution of the Navier--Stokes equation, and
suppose $h\in L^1\left(\left[0,T_{max}\right);
\mathbb{R}^+\right)$. 
Let
\begin{equation}
    f(x,t)= \begin{cases}
        \lambda_2^+(x,t), &\text{if } \lambda_2^+(x,t)> h(t)
        \\
        0, &\text{if } \lambda_2^+(x,t)\leq h(t)
            \end{cases}.
\end{equation}
If $T_{max}<+\infty,$ then
\begin{equation}
    \limsup_{t \to T_{max}}
    \left\|f(\cdot,t)\right\|_{L^\frac{3}{2}}
    \geq 3\left(\frac{\pi}{2}\right)^\frac{4}{3} \nu.
\end{equation}
\end{theorem}

We do not have a regularity criteria for 
$\lambda_2^+ \in L^\infty_T L^\frac{3}{2}_x+ L^1_T L^\infty_x$ in general, but we do have a regularity criteria in this sum space so long as the portion in $L^\infty_T L^\frac{3}{2}_x$ is small, rather than just finite. We will note that the piece of $\lambda_2^+$ in
$L^1_T L^\infty_x$ is given by
\begin{equation}
    g(x,t)= \begin{cases}
        \lambda_2^+(x,t), 
        &\text{if } \lambda_2^+(x,t)\leq h(t) \\
        0, &\text{if } \lambda_2^+(x,t)> h(t)
            \end{cases}.
\end{equation}
We will also show analogous results to Theorem \ref{EigenEndpointIntro} for $u$ and $\omega$.

\begin{theorem} \label{VelocityEndpointIntro}
Suppose $u\in C\left(\left[0,T_{max}\right);
\dot{H}^1_{df}\right)$ is a smooth solution of the Navier--Stokes equation, and
suppose $h\in L^2\left(\left[0,T_{max}\right);
\mathbb{R}^+\right)$.
Let
\begin{equation}
    v(x,t)= \begin{cases}
        u(x,t), &\text{if } |u(x,t)|> h(t)
        \\
        0, &\text{if } |u(x,t)|\leq h(t) 
        \end{cases}.
\end{equation}
If $T_{max}<+\infty,$ then
\begin{equation}
    \limsup_{t \to T_{max}}
    \left\|v(\cdot,t)\right\|_{L^3}
    \geq \sqrt{3} \left(\frac{\pi}{2}\right)^\frac{2}{3} \nu.
\end{equation}
\end{theorem}

\begin{theorem} \label{VortEndpointIntro}
Suppose $u\in C\left(\left[0,T_{max}\right);
\dot{H}^1_{df}\right)$ is a smooth solution of the Navier--Stokes equation, and
suppose $h\in L^1\left(\left[0,T_{max}\right);
\mathbb{R}^+\right)$. 
Let
\begin{equation}
    v(x,t)= \begin{cases}
        \omega(x,t), &\text{if } |\omega(x,t)|> h(t)
        \\
        0, &\text{if } |\omega(x,t)|\leq h(t)
        \end{cases}.
\end{equation}
If $T_{max}<+\infty,$ then
\begin{equation}
    \limsup_{t \to T_{max}}
    \left\|v(\cdot,t)\right\|_{L^\frac{3}{2}}
    \geq \frac{3 \pi^\frac{4}{3}}{2^\frac{5}{6}} \nu.
\end{equation}
\end{theorem}

Recently, Barker and Prange showed that if
$\left(x_0,T_{max}\right)$ 
is a singular point for a local energy solution of the Navier--Stokes equation, then 
\begin{equation}
    \|u(\cdot,t)\|_{L^3\left(B_R(x_0)\right)}
    \geq C \nu,
\end{equation}
where $R=O(\sqrt{T_{max}-t})$ 
and $C$ is a universal constant \cite{BarkerPrange}. 
Theorem \ref{VelocityEndpointIntro} can be seen as complimenting this result. Barker and Prange showed that near a singular point, the $L^3$ norm of $u$ must be bounded below when restricting to smaller and smaller neighborhoods of $x_0$ 
as $t\to T_{max}$, 
whereas Theorem \ref{VelocityEndpointIntro} requires that the $L^3$ norm of $u$ must be bounded below when restricting only to larger and larger values of $u$ as $t \to T_{max}$.
Both results give lower bounds on the concentration of critical norms near singularities: 
Barker and Prange's result gives a lower bound on the concentration of the $L^3$ norm of $u$ in the domain as a solution approaches the blowup time, whereas Theorem \ref{VelocityEndpointIntro} gives a lower bound on the concentration of the $L^3$ norm in the range
as $t \to T_{max}$.

While it remains an open question whether 
$\|\lambda_2^+(\cdot,t)\|_{L^\frac{3}{2}}$
must blow up as $t \to T_{max}$ if $T_{max}<+\infty$, 
Theorem \ref{EigenEndpointIntro} and some further analysis that we will discuss in section \ref{EigenSection} suggest that this norm must blowup in order for a smooth solution of the Navier--Stokes equation to develop singularities in finite-time.

\begin{conjecture} \label{EigenEndpointConjecture}
Suppose $u\in C\left(\left[0,T_{max}\right);
\dot{H}^1_{df}\right)$ is a smooth solution of the Navier--Stokes equation, and $T_{max}<+\infty.$
Then
\begin{equation}
    \limsup_{t\to T_{max}}
    \|\lambda_2^+(\cdot,t)\|_{L^\frac{3}{2}}
    =+\infty.
\end{equation}
\end{conjecture}

\begin{remark}
Using Theorem \ref{MixedLpSumIntro}, we can strengthen the regularity criteria in Theorems \ref{VelocitySumSpaceIntro}, \ref{StrainSumSpaceIntro}, and \ref{VortSumSpaceIntro}, by further enlarging the space. We can relax the control required in Theorem \ref{VelocitySumSpaceIntro} from the space $L^p_T L^q_x+ L^2_T L^\infty_x$ to the slightly larger space
$L^p_T L^{q,\infty}_x+ L^2_T L^\infty_x.$
Likewise we can relax the control required in Theorems \ref{StrainSumSpaceIntro} and \ref{VortSumSpaceIntro} from the space $L^p_T L^q_x+ L^1_T L^\infty_x$
to the slightly large space
$L^p_T L^{q,\infty}_x+ L^1_T L^\infty_x.$
These corollaries are stated below.
\end{remark}

\begin{corollary} \label{VelocitySumSpaceWeakIntro}
Suppose $u\in C\left(\left[0,T_{max}\right);
\dot{H}^1_{df}\right)$ is a smooth solution of the Navier--Stokes equation. Let $3<q<+\infty,
\frac{2}{p}+\frac{3}{q}=1,$ and let 
$ u=v+\sigma.$ Then for all $0<T<T_{max}$
\begin{equation} 
    \|\nabla u(\cdot,T)\|_{L^2}^2 \leq 
    \left\| \nabla u^0 \right\|_{L^2}^2
    \exp\left(\Tilde{C}_p \int_0^T 
    \|v(\cdot,t)\|_{L^{q,\infty}}^p \diff t
    +\frac{2}{\nu} \int_0^T \|\sigma(\cdot,t)\|_{L^\infty}^2
    \diff t \right),
\end{equation}
where
\begin{equation}
    \Tilde{C}_p=
    \frac{C_{p'}}{\nu^{p'-1}}
    \left(\frac{q}{q-q'}\right)^\frac{p'}{q'}
    +\frac{2}{\nu},
\end{equation}
with $C_{p'}$ is taken as in 
Theorem \ref{VelocitySumSpaceIntro},
and $3<q'<q, \frac{2}{p'}+\frac{3}{q'}=1.$
In particular if $T_{max}<+\infty,$ then
\begin{equation}
    \Tilde{C}_p \int_0^T 
    \|v(\cdot,t)\|_{L^{q,\infty}}^p \diff t
    +\frac{2}{\nu} \int_0^T \|\sigma(\cdot,t)\|_{L^\infty}^2
    \diff t 
    =+\infty.
\end{equation}
\end{corollary}

\begin{corollary} \label{StrainSumSpaceWeakIntro}
Suppose $u\in C\left(\left[0,T_{max}\right);
\dot{H}^1_{df}\right)$ is a smooth solution of the Navier--Stokes equation. 
Let $\frac{3}{2}<q<+\infty, \frac{2}{p}+\frac{3}{q}=2,$ 
and let 
$\lambda_2^+ =f+g.$ 
Then for all $0<T<T_{max}$
\begin{equation}
    \|S(\cdot,T)\|_{L^2}^2 \leq 
    \left\| S^0 \right\|_{L^2}^2
    \exp\left( \Tilde{C}_p
    \int_0^T \|f(\cdot,t)\|_{L^{q,\infty}}^p \diff t
    +2\int_0^T \|g(\cdot,t)\|_{L^\infty} \diff t
    \right),
\end{equation}
where
\begin{equation}
    \Tilde{C}_p=
    \frac{C_{p'}}{\nu^{p'-1}}
    \left(\frac{q}{q-q'}\right)^\frac{p'}{q'} +2,
\end{equation}
with $C_{p'}$ is taken as in 
Theorem \ref{StrainSumSpaceIntro},
and $\frac{3}{2}<q'<q, \frac{2}{p'}+\frac{3}{q'}=2.$
In particular if $T_{max}<+\infty,$ then
\begin{equation}
    \Tilde{C}_p \int_0^T 
    \|f(\cdot,t)\|_{L^{q,\infty}}^p \diff t
    +2\int_0^T \|g(\cdot,t)\|_{L^\infty} \diff t
    =+\infty.
\end{equation}
\end{corollary}

\begin{corollary} \label{VortSumSpaceWeakIntro}
Suppose $u\in C\left(\left[0,T_{max}\right);
\dot{H}^1_{df}\right)$ is a smooth solution of the Navier--Stokes equation. 
Let $\frac{3}{2}<q<+\infty, \frac{2}{p}+\frac{3}{q}=2,$ 
and let 
$\omega=v+\sigma$. 
Then for all $0<T<T_{max},$
\begin{equation}
    \|\omega(\cdot,T)\|_{L^2}^2 \leq 
    \left\| \omega^0 \right\|_{L^2}^2
    \exp\left(\Tilde{C_p} \int_0^T 
    \|v(\cdot,t)\|_{L^{q,\infty}}^p \diff t
    +\sqrt{2}\int_0^T \|\sigma(\cdot,t)\|_{L^\infty} 
    \diff t \right),
\end{equation}
where
\begin{equation}
    \Tilde{C}_p=
    \frac{C_{p'}}{\nu^{p'-1}}
    \left(\frac{q}{q-q'}\right)^\frac{p'}{q'}
    +\sqrt{2},
\end{equation}
with $C_{p'}$ is taken as in 
Theorem \ref{VortSumSpaceIntro},
and $\frac{3}{2}<q'<q, \frac{2}{p'}+\frac{3}{q'}=2.$
In particular if $T_{max}<+\infty,$ then
\begin{equation}
    \Tilde{C_p} \int_0^T 
    \|v(\cdot,t)\|_{L^{q,\infty}}^p \diff t
    +\sqrt{2}\int_0^T \|\sigma(\cdot,t)\|_{L^\infty} 
    \diff t
    =+\infty.
\end{equation}
\end{corollary}

In section \ref{Defintions}, we will define our notation and the main spaces used in the paper, and we will state
the precise definition of mild solutions,
as well as some of the classical results that we will use in the paper.
In section \ref{EigenSection}, we will consider regularity criteria in sum spaces in terms of $\lambda_2^+$, proving Theorems \ref{StrainSumSpaceIntro} and \ref{EigenEndpointIntro}.
In section \ref{SumSpaceSection}, we will discuss the structure of mixed Lebesgue sum spaces, proving Theorem \ref{MixedLpSumIntro}, and we will also introduce the distribution function and weak $L^q$, proving a number of the core properties.
In section \ref{VelocitySection}, we will consider regularity criteria in sum spaces in terms of $u$, 
proving Theorems \ref{VelocitySumSpaceIntro} and \ref{VelocityEndpointIntro}.
In section \ref{VorticitySection}, we will consider regularity criteria in sum spaces in terms of $\omega$, 
proving Theorems \ref{VortSumSpaceIntro} and \ref{VortEndpointIntro}.

\section{Definitions and notation} \label{Defintions}

Before proceeding with the proofs of our results, we need to define a number of spaces.
First we will define the inhomogeneous Hilbert spaces on $\mathbb{R}^3.$

\begin{definition}
For all $\alpha \in \mathbb{R},$ let
\begin{equation}
    \|u\|_{H^\alpha}^2=
    \int_{\mathbb{R}^3} \left(1+4\pi^2 |\xi|^2 \right)^\alpha
    \left|\hat{u}(\xi)\right|^2 \diff\xi,
\end{equation}
and let 
\begin{equation}
    H^\alpha \left(\mathbb{R}^3\right)
    =
    \left\{u\in \mathcal{S}'\left(\mathbb{R}^3\right):
    \|u\|_{H^\alpha}<+\infty\right\},
\end{equation}
where $\mathcal{S}'\left(\mathbb{R}^3\right)$ 
is the space of tempered distributions.
\end{definition}

We have defined the space 
$H^\alpha\left(\mathbb{R}^3\right)$; 
now we will define the space
$\dot{H}^\alpha\left(\mathbb{R}^3\right)$.

\begin{definition}
For all $\alpha \in \mathbb{R},$ let
\begin{equation}
    \|u\|_{\dot{H}^\alpha}^2=
    \int_{\mathbb{R}^3} (2\pi)^{2\alpha} |\xi|^{2\alpha}
    \left|\hat{u}(\xi)\right|^2 \diff\xi,
\end{equation}
and let 
\begin{equation}
    \dot{H}^\alpha \left(\mathbb{R}^3\right)
    =
    \left\{u\in \mathcal{S}'\left(\mathbb{R}^3\right):
    \|u\|_{\dot{H}^\alpha}<+\infty\right\}.
\end{equation}
\end{definition}

Note that $H^\alpha\left(\mathbb{R}^3\right)$ is a Hilbert space 
for all $\alpha\in\mathbb{R},$
while $\dot{H}\left(\mathbb{R}^3\right)$ is a Hilbert space
for all $-\frac{3}{2}<\alpha<\frac{3}{2},$ although is still well defined outside of this range.
We will further note that 
for all $u\in \dot{H}^1\left(\mathbb{R}^3\right)$
\begin{equation}
    \|u\|_{\dot{H}^1}=\|\nabla u\|_{L^2},
\end{equation}
and for all $u\in H^1\left(\mathbb{R}^3\right)$
\begin{equation}
    \|u\|_{H^1}^2=\|u\|_{L^2}^2+\|\nabla u\|_{L^2}^2
\end{equation}
Another property of $\dot{H}^1$ is the Sobolev embedding
$\dot{H}^1\left(\mathbb{R}^3\right) 
\hookrightarrow L^6\left(\mathbb{R}^3\right),$
and the related Sobolev inequality.

\begin{theorem} \label{Sobolev}
 For all $f\in L^6\left(\mathbb{R}^3\right),$
 \begin{equation}
     \|f\|_{L^6} \leq 
     \frac{1}{\sqrt{3}}\left(
    \frac{2}{\pi}\right)^\frac{2}{3}
    \|\nabla f\|_{L^2}.
 \end{equation}
\end{theorem}
Theorem \ref{Sobolev} was first proven by Sobolev in \cite{Sobolev}, and the sharp version of this inequality was proven by Talenti \cite{Talenti}. For a thorough reference on this inequality and certain generalizations, see also \cite{LiebLoss}.
The Sobolev inequality will play an essential role in the proof of each of the regularity criterion, by allowing us to make use of the dissipation due to viscosity in controlling the solution.

Next we will define the subspaces of divergence free vector fields in the spaces $\dot{H}^\alpha\left(\mathbb{R}^3;\mathbb{R}^3\right)$
and $H^\alpha\left(\mathbb{R}^3;\mathbb{R}^3\right)$
This is useful because by building the divergence free constraint, $\nabla \cdot u=0,$ into our function space, we can treat the Navier--Stokes equation as an evolution equation on this function space and not a system of equations. We will do this by expressing condition $\nabla \cdot u=0$ in Fourier space, where it can be written as 
$\xi \cdot \hat{u}(\xi)=0.$

\begin{definition}
For all $\alpha \in \mathbb{R}$
\begin{equation}
    \dot{H}^\alpha_{df}=
    \left\{ u\in \dot{H}^\alpha\left(
    \mathbb{R}^3;\mathbb{R}^3\right):
    \xi \cdot \hat{u}(\xi)=0,
    \text{almost everywhere } \xi\in\mathbb{R}^3\right\}.
\end{equation}
\end{definition}

\begin{definition}
For all $\alpha \in \mathbb{R}$
\begin{equation}
    H^\alpha_{df}=
    \left\{ u\in H^\alpha\left(
    \mathbb{R}^3;\mathbb{R}^3\right):
    \xi \cdot \hat{u}(\xi)=0,
    \text{almost everywhere } \xi\in\mathbb{R}^3\right\}.
\end{equation}
\end{definition}

Now that we have defined the space $\dot{H}^1_{df}$, we will give the precise definition of a mild solution, developed by Fujita and Kato in \cite{KatoFujita}.

\begin{definition}
\label{MildSolutions}
Suppose $u \in C\left([0,T);\dot{H}^1_{df} \right ).$
Then $u$ is a mild solution to the Navier--Stokes equation if for all $0\leq t<T$
\begin{equation}
u(\cdot,t)=e^{\nu t \Delta}u^0
+\int_0^t e^{\nu(t-\tau)\Delta}
P_{df}\left(-(u \cdot \nabla)u\right)
(\cdot, \tau)\diff \tau,
\end{equation}
where $e^{t\Delta}$ is the operator associated with the heat semi-group given by convolution with the heat kernel. 
\end{definition}

Note that by using the projection $P_{df}$ onto the space $\dot{H}^1_{df},$ we are able to build the divergence free constraint into the definition of the solution without treating it as a separate equation to satisfy, and in particular without any need to make reference to the pressure. 
Kato and Fujita also proved the local in time existence of mild solutions, as well as their uniqueness and higher regularity. The proof is based on a Picard iteration scheme using the heat kernel, and the argument can only be made to close when $T$ is sufficiently small in terms of 
$\left\|u^0\right\|_{\dot{H}^1}$. The precise statement of their result is as follows.

\begin{theorem} \label{MildExistence}
For all $u^0\in \dot{H}^1_{df},$
there exists a unique mild solution to the Navier Stokes equation $u \in C\left([0,T);
\dot{H^1}_{df} \right),$
$u(\cdot,0)=u^0,$
where $T=\frac{C \nu^3}{||u^0||_{\dot{H^1}}^4}$,
and $C$ is an absolute constant 
independent of $u$ and $\nu.$
Furthermore, this solution will have higher regularity, 
$u \in C^\infty\left ((0,T)\times \mathbb{R}^3\right ).$ 
\end{theorem}

Note that if we take $u(\cdot,t),$ as initial data, the uniqueness result in Theorem \ref{MildExistence} combined with the lower bound on the time of existence, implies that 
\begin{equation}
    T_{max}-t \geq \frac{C \nu^3}
    {\|u(\cdot,t)\|_{\dot{H}^1}^4},
\end{equation}
and therefore if $T_{max}<+\infty,$
then for all $0\leq t<T_{max}$
\begin{equation} \label{BlowupEnstrophyLB}
    \|u(\cdot,t)\|_{\dot{H}^1}^4 \geq
    \frac{C \nu^3}{T_{max}-t}.
\end{equation}

We also need to define the mixed Lebesgue space $L^p_T L^q_x$.
\begin{definition}
for all $1\leq p,q \leq +\infty,$
\begin{equation}
    L^p_T L^q_x= L^p\left([0,T);
    L^q\left(\mathbb{R}^3\right)\right).
\end{equation}
For $1\leq p<+\infty$
\begin{equation}
    \|f\|_{L^p_T L^q_x}=
    \left(\int_0^T \|f(\cdot,t)\|_{L^q}^p 
    \diff t\right)^{\frac{1}{p}},
\end{equation}
and for $p=\infty,$
\begin{equation}
    \|f\|_{L^\infty_T L^q_x}=
    \esssup_{0\leq t<T} \|f(\cdot,t)\|_{L^q}.
\end{equation}
\end{definition}
Note that throughout the paper we will often drop the $\mathbb{R}^3$
when referring to $L^q \left(\mathbb{R}^3\right)$ or 
$\dot{H}^1 \left(\mathbb{R}^3\right)$. We will sometimes use the notation
\begin{equation}
    L^q_x=L^q \left(\mathbb{R}^3\right),
\end{equation}
when necessary for clarity in cases where both spatial and time variables are involved.

Additionally, we must define sum spaces, which play such an essential role our in results.
\begin{definition}
Let $X$ and $Y$ be Banach spaces, 
and let $V$ be a vector space with $X,Y \subset V.$ 
Then
\begin{equation}
    X+Y=\left\{x+y: x\in X, y\in Y\right\}.
\end{equation}
Furthermore, $X+Y$ is a Banach space with norm
\begin{equation}
    \|f\|_{X+Y}= \inf_{g+h=f} \|g\|_X+ \|h\|_Y.
\end{equation}
\end{definition}

Finally we will define enstrophy.
\begin{definition}
Let $u\in C\left([0,T_{max};\dot{H}^1_{df}\right)$ be a mild solution of the Navier--Stokes equation. Then for all $0\leq t<T$ the enstrophy is given by
\begin{align}
    \mathcal{E}(t)
    &=
    \frac{1}{2}\|\omega(\cdot,t)\|_{L^2}^2 \\
    &=
    \frac{1}{2}\|\nabla u(\cdot,t)\|_{L^2}^2 \\
    &=
    \|S\|_{L^2}^2.
\end{align}
\end{definition}

The enstrophy plays an important role in the Navier--Stokes equation, because as we showed in \eqref{BlowupEnstrophyLB},
if $T_{max}<+\infty,$
then for all $0\leq t<T_{max}$
\begin{equation}
    \|u(\cdot,t)\|_{\dot{H}^1}^4 \geq
    \frac{C \nu^3}{T_{max}-t}.
\end{equation}
In particular, this means that
if $T_{max}<+\infty,$ then
\begin{align}
    \lim_{t\to T_{max}} \mathcal{E}(t)
    &=
    \lim_{t\to T_{max}} 
    \frac{1}{2}\|u(\cdot,t)\|_{\dot{H}^1}^2 \\
    &=
    +\infty.
\end{align}
Consequently, the proofs of all of our regularity criteria will rely on estimates for the growth of enstrophy defined in terms of 
$S, u,$ or $\omega$, because it is sufficient to control enstrophy up until some time $T$,
to guarantee that a smooth solution can be continued to some time $\Tilde{T}>T.$ 

Note that the various definitions of enstrophy 
in terms of $S, u,$ and $\omega$ are equivalent due to an isometry for the strain, vorticity and gradient of divergence free vector fields proven by the author in \cite{MillerStrain}.

\begin{proposition} \label{isometry}
For all $u\in \dot{H}^{\alpha+1}_{df},$
\begin{align}
    \|\nabla u\|_{\dot{H}^\alpha}^2
    &=
    \|\omega\|_{\dot{H}^\alpha}^2\\
    &=
    \frac{1}{2}\|S\|_{\dot{H}^\alpha}^2.
\end{align}
\end{proposition}

We have now introduced all the spaces that we will use in this paper with the exception of $L^{q,\infty}$. We will leave the definition of this space until section \ref{SumSpaceSection}, where it fits more naturally, in order to keep the presentation of the sum space inclusion in Theorem \ref{MixedLpSumIntro} self-contained.

\section{Middle eigenvalue regularity criterion} \label{EigenSection}

In this section we will consider regularity criteria for $\lambda_2^+$ in sum spaces of scale invariant spaces.
We will begin by recalling an estimate for enstrophy growth proven by the author in \cite{MillerStrain}, variants of which were also considered in \cites{NeustupaPenel1,ChaeStrain}.

\begin{proposition} \label{EnstrophyGrowth}
Suppose $u\in C\left(\left[0,T_{max}\right);
\dot{H}^1_{df}\right)$ is a smooth solution of the Navier--Stokes equation. Then for all $0<t<T_{max}$
\begin{align}
    \partial_t\|S(\cdot,t)\|_{L^2}^2
    &=
    -2 \nu\|S\|_{\dot{H}^1}^2- 4\int_{\mathbb{R}^3} \det(S)\\
    &\leq
    -2 \nu \|S\|_{\dot{H}^1}^2+ 
    2\int_{\mathbb{R}^3} \lambda_2^+ |S|^2,
\end{align}
where $\lambda_1(x,t)\leq \lambda_2(x,t)\leq \lambda_3(x,t)$
are the eigenvalues of $S(x,t)$,
and $\lambda_2^+ =\max\left(0,\lambda_2\right)$.
\end{proposition}

Using this estimate for enstrophy growth, we will prove Theorem \ref{StrainSumSpaceIntro}, which is restated here for the reader's convenience.

\begin{theorem} \label{StrainSumSpace}
Suppose $u\in C\left(\left[0,T_{max}\right);
\dot{H}^1_{df}\right)$ is a smooth solution of the Navier--Stokes equation. 
Let $\frac{3}{2}<q<+\infty, \frac{2}{p}+\frac{3}{q}=2,$ 
and let 
$\lambda_2^+ =f+g.$ 
Then for all $0<T<T_{max}$
\begin{equation} \label{EnstrophyBoundLambda}
    \|S(\cdot,T)\|_{L^2}^2 \leq 
    \left\| S^0 \right\|_{L^2}^2
    \exp\left(\frac{C_p}{\nu^{p-1}} \int_0^T 
    \|f(\cdot,t)\|_{L^q}^p \diff t
    +2\int_0^T \|g(\cdot,t)\|_{L^\infty} \diff t
    \right),
\end{equation}
where $C_p$ depends only on $p.$
In particular if $T_{max}<+\infty,$ then
\begin{equation}
    \frac{C_p}{\nu^{p-1}} \int_0^{T_{max}} 
    \|f(\cdot,t)\|_{L^q}^p \diff t
    +2\int_0^{T_{max}} \|g(\cdot,t)\|_{L^\infty} \diff t
    =+\infty.
\end{equation}
\end{theorem}

\begin{proof}
We know that if $T_{max}<+\infty,$ then 
\begin{equation}
    \lim_{T \to T_{max}}\|S(\cdot,T)\|_{L^2}^2=+\infty,    
\end{equation}
so it suffices to prove the bound \eqref{EnstrophyBoundLambda}.
We can see from the inequality in 
Proposition \ref{EnstrophyGrowth}, that for all $0<t<T_{max},$
\begin{align}
    \partial_t \|S(\cdot,t)\|_{L^2}^2 
    &\leq
    -2 \nu \|S\|_{\dot{H}^1}^2 
    +2\int_{\mathbb{R}^3} \lambda_2^+ |S|^2 \\
    &=
    -2 \nu \|S\|_{\dot{H}^1}^2 
    +2\int_{\mathbb{R}^3} (f+g) |S|^2 \\
    &\leq
    -2 \nu \|S\|_{\dot{H}^1}^2+
    2\|f\|_{L^q}\left\||S|^2\right\|_{L^{r}}
    +2\|g\|_{L^\infty}\left\||S|^2\right\|_{L^1}\\
    &= \label{StepA}
    -2 \nu \|S\|_{\dot{H}^1}^2+
    2\|f\|_{L^q}\|S\|_{L^{2r}}^2+2\|g\|_{L^\infty}\|S\|_{L^2}^2,
\end{align}
where $\frac{1}{r}+\frac{1}{q}=1,$ and we have applied H\"older's inequality with exponents $q,r$ and $1,\infty.$

Next we observe that $\frac{3}{2}<q<\infty,$ and so
$1<r<3,$ and consequently $2<2r<6.$ 
Let $\rho=\frac{3}{2q}.$
We can see that $0<\rho<1,$ and
\begin{align}
    (1-\rho) \frac{1}{2}+ \rho \frac{1}{6}
    &=
    \frac{1}{2}-\frac{\rho}{3}\\
    &=
    \frac{1}{2}-\frac{1}{2q}\\
    &=
    \frac{1}{2}-\frac{1}{2}\left(1-\frac{1}{r}\right)\\
    &=
    \frac{1}{2r}.
\end{align}
Therefore, we can interpolate between $L^2$ and $L^6$ and find that
\begin{equation}
    \|S\|_{L^{2r}}\leq 
    \|S\|_{L^2}^{1-\frac{3}{2q}} \|S\|_{L^6}^{\frac{3}{2q}}.
\end{equation}
Plugging back into \eqref{StepA}, and applying the Sobolev inequality (Theorem \ref{Sobolev}), we find that
\begin{align}
    \partial_t \|S(\cdot,t)\|_{L^2}^2 
    &\leq
    -2 \nu \|S\|_{\dot{H}^1}^2+
    2\|f\|_{L^q}\|S\|_{L^2}^{2-\frac{3}{q}} 
    \|S\|_{L^6}^{\frac{3}{q}}
    +2\|g\|_{L^\infty}\|S\|_{L^2}^2 \\
    &\leq
    -2 \nu \|S\|_{\dot{H}^1}^2+
    C \|f\|_{L^q}\|S\|_{L^2}^{2-\frac{3}{q}} 
    \|S\|_{\dot{H}^1}^{\frac{3}{q}}
    +2\|g\|_{L^\infty}\|S\|_{L^2}^2 \\
    &= 
    -2 \nu \|S\|_{\dot{H}^1}^2+
    C \|f\|_{L^q}\|S\|_{L^2}^{\frac{2}{p}} 
    \|S\|_{\dot{H}^1}^{\frac{3}{q}}
    +2\|g\|_{L^\infty}\|S\|_{L^2}^2 
\end{align}
Let $b=\frac{2q}{3}.$ Clearly $1<b<+\infty,$ and recalling that 
$\frac{2}{p}+\frac{3}{q}=2,$ we can see that
\begin{align}
    \frac{1}{p}+\frac{1}{b}
    &=
    \frac{1}{p}+\frac{3}{2q}\\
    &=
    1.
\end{align}
Applying Young's inequality with exponents $p,b$ we find
\begin{equation}
    \frac{C}{\nu} \|f\|_{L^q}\|S\|_{L^2}^{\frac{2}{p}} 
    \|S\|_{\dot{H}^1}^{\frac{3}{q}}
    \leq
    \frac{C_p}{\nu^p}\|f\|_{L^q}^p\|S\|_{L^2}^2
    +2\|S\|_{\dot{H}^1}^2.
\end{equation}
This immediately implies that
\begin{equation}
    -2 \nu \|S\|_{\dot{H}^1}^2+
    C \|f\|_{L^q}\|S\|_{L^2}^\frac{2}{p} 
    \|S\|_{\dot{H}^1}^{\frac{3}{q}}
    \leq
    \frac{C_p}{\nu^{p-1}}\|f\|_{L^q}^p \|S\|_{L^2}^2.
\end{equation}
Note that while we are not keeping track of the value of the constant $C_p$, it is nevertheless independent of $\nu$, and is determined solely in terms of $p$ and the value of the sharp Sobolev constant.
From this we may conclude that
\begin{equation}
    \partial_t \|S(\cdot,t)\|_{L^2}^2
    \leq
    \left(\frac{C_p}{\nu^{p-1}}\|f\|_{L^q}^p
    +2 \|g\|_{L^\infty}\right) \|S\|_{L^2}^2.
\end{equation}
Applying Gr\"onwall's inequality we find that for all $0<T<T_{max},$
\begin{equation} 
    \|S(\cdot,T)\|_{L^2}^2 \leq 
    \left\| S^0 \right\|_{L^2}^2
    \exp\left(\int_0^T 
    \left(\frac{C_p}{\nu^{p-1}}\|f(\cdot,t)\|_{L^q}^p
    +2 \|g(\cdot,t)\|_{L^\infty}\right) \diff t
    \right),
\end{equation}
and this completes the proof.
\end{proof}

It is clear in general that the $L^\infty_x$ norm is effective for controlling values of $\lambda_2^+(x,t)$ in the large regions of $\mathbb{R}^3$ where it is relatively small. Therefore, Theorem \ref{StrainSumSpace}, implies that the $L^p_T L^q_x,$ norm must be large in the small regions of space where $\lambda_2^+(x,t)$ takes large values, and that in this sense $\lambda_2^+$ must exhibit concentrated blowup in the critical norms $L^p_T L^q_x$, for all $\frac{3}{2}<q<+\infty.$ 
We will prove a corollary that quantifies this phenomenon, requiring the concentration of the $L^p_T L^q_x$ norm at large values in the range, for all 
$\frac{3}{2}<q<+\infty, \frac{2}{p}+\frac{3}{q}=2$.

\begin{corollary} \label{StrainCor}
Suppose $u\in C\left(\left[0,T_{max}\right);
\dot{H}^1_{df}\right)$ is a smooth solution of the Navier--Stokes equation, and
suppose $h\in L^1\left(\left[0,T_{max}\right);
\mathbb{R}^+\right)$. 
Let $\frac{3}{2}<q<+\infty,
\frac{2}{p}+\frac{3}{q}=2,$ and let
\begin{equation}
    f(x,t)= \begin{cases}
        \lambda_2^+(x,t), &\text{if } \lambda_2^+(x,t)> h(t)
        \\
        0, &\text{if } \lambda_2^+(x,t)\leq h(t) 
            \end{cases}.
\end{equation}
Then for all $0<T<T_{max}$
\begin{equation} \label{EnstrophyBoundLambdaCor}
    \|S(\cdot,T)\|_{L^2}^2 \leq 
    \left\| S^0 \right\|_{L^2}^2
    \exp\left(\frac{C_p}{\nu^{p-1}} \int_0^T 
    \|f(\cdot,t)\|_{L^q}^p \diff t
    +2\int_0^T  h(t) \diff t
    \right),
\end{equation}
where $C_p$ depends only on $p.$
In particular if $T_{max}<+\infty,$ then
\begin{equation}
    \int_0^{T_{max}} 
    \|f(\cdot,t)\|_{L^q}^p \diff t
    =+\infty.
\end{equation}
\end{corollary}

\begin{proof}
We will begin by letting
\begin{equation}
    g(x,t)= \begin{cases}
        \lambda_2^+(x,t), 
        &\text{if } \lambda_2^+(x,t)\leq h(t) \\
        0, &\text{if } \lambda_2^+(x,t)> h(t)
            \end{cases}
\end{equation}
We can see immediately that for all $0<t<T_{max},$
\begin{equation}
   \|g(\cdot,t)\|_{L^\infty}\leq h(t), 
\end{equation}
and that
\begin{equation}
    \lambda_2^+=f+g.
\end{equation}
Therefore we can apply Theorem \ref{StrainSumSpace} and find that
\begin{align}
    \|S(\cdot,T)\|_{L^2}^2 
    &\leq 
    \left\| S^0 \right\|_{L^2}^2
    \exp\left(\frac{C_p}{\nu^{p-1}} \int_0^T 
    \|f(\cdot,t)\|_{L^q}^p \diff t
    +2\int_0^T \|g(\cdot,t)\|_{L^\infty} \diff t
    \right)\\
    &\leq
    \left\| S^0 \right\|_{L^2}^2
    \exp\left(\frac{C_p}{\nu^{p-1}} \int_0^T 
    \|f(\cdot,t)\|_{L^q}^p \diff t
    +2\int_0^T  h(t) \diff t
    \right).
\end{align}

Next we will note, as in Theorem \ref{StrainSumSpace}, that if $T_{max}<+\infty,$ then
\begin{equation}
    \lim_{T\to T_{max}}\|S(\cdot,T)\|_{L^2}^2=+\infty.
\end{equation}
Therefore we can conclude that if $T_{max}<+\infty,$ then
\begin{equation}
    \frac{C_p}{\nu^{p-1}} \int_0^{T_{max}} 
    \|f(\cdot,t)\|_{L^q}^p \diff t
    +2\int_0^{T_{max}}  h(t) \diff t
    =+\infty.
\end{equation}
However, we know by hypothesis that
\begin{equation}
    \int_0^{T_{max}}  h(t) \diff t<+\infty,
\end{equation}
so we may conclude that
\begin{equation}
    \int_0^{T_{max}} 
    \|f(\cdot,t)\|_{L^q}^p \diff t
    =+\infty.
\end{equation}
This completes the proof.
\end{proof}

This concentrated blowup in $L^p_T L^q_x,$ for $\frac{2}{p}+\frac{3}{q}=2,$ with $q>\frac{3}{2},$ arbitrarily close to $\frac{3}{2}$, heavily suggests that if $T_{max}<+\infty,$
then
\begin{equation}
    \limsup_{t \to T_{max}}
    \left\|\lambda_2^+(\cdot,t)\right\|_{L^\frac{3}{2}}
    =+\infty,
\end{equation}
and so Conjecture \ref{EigenEndpointConjecture} holds,
although to establish this result is still beyond the scope of the methods used in this paper.
In \cite{MillerStrain}, the author showed that 
If $T_{max}<+\infty,$ then
\begin{equation}
    \limsup_{t \to T_{max}}
    \left\|\lambda_2^+(\cdot,t)\right\|_{L^\frac{3}{2}}
    \geq 3\left(\frac{\pi}{2}\right)^\frac{4}{3} \nu.
\end{equation}
We will prove the sum space analogue of this result now, which is also the endpoint case of Corollary \ref{StrainCor}.

\begin{theorem} \label{EigenEndpoint}
Suppose $u\in C\left(\left[0,T_{max}\right);
\dot{H}^1_{df}\right)$ is a smooth solution of the Navier--Stokes equation, and
suppose $h\in L^1\left(\left[0,T_{max}\right);
\mathbb{R}^+\right)$. 
Let
\begin{equation}
    f(x,t)= \begin{cases}
        \lambda_2^+(x,t), &\text{if } \lambda_2^+(x,t)> h(t)
        \\
        0, &\text{if } \lambda_2^+(x,t)\leq h(t)
            \end{cases}.
\end{equation}
If $T_{max}<+\infty,$ then
\begin{equation}
    \limsup_{t \to T_{max}}
    \left\|f(\cdot,t)\right\|_{L^\frac{3}{2}}
    \geq 3\left(\frac{\pi}{2}\right)^\frac{4}{3} \nu.
\end{equation}
\end{theorem}

\begin{proof}
Suppose toward contradiction that $T_{max}<+\infty$ and
\begin{equation}
    \limsup_{t \to T_{max}}
    \left\|f(\cdot,t)\right\|_{L^\frac{3}{2}}
    < 3\left(\frac{\pi}{2}\right)^\frac{4}{3} \nu.
\end{equation}
This implies that there exists $\epsilon>0,$ such that for all $T_{max}-\epsilon<t<T_{max},$
\begin{equation}
    \left\|f(\cdot,t)\right\|_{L^\frac{3}{2}}
    < 3\left(\frac{\pi}{2}\right)^\frac{4}{3} \nu.
\end{equation}
We will again let
\begin{equation}
    g(x,t)= \begin{cases}
        \lambda_2^+(x,t), 
        &\text{if } \lambda_2^+(x,t)\leq h(t) \\
        0, &\text{if } \lambda_2^+(x,t)> h(t)
            \end{cases}
\end{equation}
We can see immediately that for all $0<t<T_{max},$
\begin{equation}
   \|g(\cdot,t)\|_{L^\infty}\leq h(t), 
\end{equation}
and that
\begin{equation}
    \lambda_2^+=f+g.
\end{equation}

Using the estimate for enstrophy growth in Proposition \ref{EnstrophyGrowth}, H\"older's inequality, and Sobolev's inequality, we find that for all $T_{max}-\delta<t<T_{max}$
\begin{align}
    \partial_t\|S(\cdot,t)\|_{L^2}^2
    &\leq
    -2 \nu \|S\|_{\dot{H}^1}^2+ 
    2\int_{\mathbb{R}^3} \lambda_2^+ |S|^2\\
    &=
    -2 \nu \|S\|_{\dot{H}^1}^2
    +2\int_{\mathbb{R}^3} f |S|^2
    +2\int_{\mathbb{R}^3} g |S|^2 \\
    &\leq 
    -2 \nu \|S\|_{\dot{H}^1}^2
    +2\|f\|_{L^\frac{3}{2}}\|S\|_{L^6}^2
    +2\|g\|_{L^\infty}\|S\|_{L^2}^2 \\
    &\leq
    -2 \nu \|S\|_{\dot{H}^1}^2
    +\frac{2}{3 \left(\frac{\pi}{2}\right)^\frac{4}{3}}
    \|f\|_{L^\frac{3}{2}}\|S\|_{\dot{H}^1}^2
    +2h \|S\|_{L^2}^2 \\
    &=
    2 \nu \|S\|_{\dot{H}^1}^2 \left(
    -1+ \frac{\|f\|_{L^\frac{3}{2}}}
    {3 \left(\frac{\pi}{2} \right)^\frac{4}{3}\nu}\right)
    +2 h \|S\|_{L^2}^2.
\end{align}
Recall that by hypothesis for all 
$T_{max}-\epsilon<t<T_{max},$
\begin{equation}
    \frac{\|f\|_{L^\frac{3}{2}}}
    {3 \left(\frac{\pi}{2} \right)^\frac{4}{3}\nu}<1,
\end{equation}
so we can conclude that for all
$T_{max}-\epsilon<t<T_{max},$
\begin{equation}
    \partial_t\|S(\cdot,t)\|_{L^2}^2
    \leq
    2 h \|S\|_{L^2}^2.
\end{equation}

Applying Gr\"onwall's inequality we can see that for all
$T_{max}-\epsilon<T<T_{max}$
\begin{equation}
    \|S(\cdot,T)\|_{L^2}^2
    \leq
    \left\|S(\cdot,T_{max}-\epsilon) \right\|_{L^2}^2
    \exp\left( 2 \int_{T_{max}-\epsilon}^T 
    h(t) \diff t\right).
\end{equation}
Recalling that 
$h\in L^1 \left([0,T_{max};\mathbb{R}^+\right),$ we can see that
\begin{align}
    \limsup_{T \to T_{max}} \|S(\cdot,T)\|_{L^2}^2
    &\leq
    \left\|S(\cdot,T_{max}-\epsilon) \right\|_{L^2}^2
    \exp\left(2 \int_{T_{max}-\epsilon}^{T_{max}} 
    h(t) \diff t\right)\\
    &<
    +\infty.
\end{align}
This contradicts our assumption that $T_{max}<+\infty,$
so this completes the proof.
\end{proof}

Before we move on to considering the structure of the mixed Lebesgue sum spaces, we will show that if $T_{max}<+\infty,$
and 
\begin{equation}
    \limsup_{t\to T_{max}}
    \|f(\cdot,t)\|_{L^\frac{3}{2}}
    <+\infty,
\end{equation}
then we can conclude that
$f(\cdot,t) \rightharpoonup 0$ weakly in $L^\frac{3}{2}$ 
as $t \to T_{max},$
where $f$ is $\lambda_2^+$ restricted to the points in its domain where it takes large values.
We will need to establish the following proposition before we can prove this statement.

\begin{proposition} \label{LqProp}
Suppose $f\in L^q\left(\mathbb{R}^3\right), 1<q<+\infty$ and there exists $R>0,$ such that for all $x\in \mathbb{R}^3,$ 
$f(x)=0$ or $|f(x)|> R.$ Then for all $1\leq p<q,$
\begin{equation}
    \|f\|_{L^p}^p \leq \frac{1}{R^{q-p}}\|f\|_{L^q}^q.
\end{equation}
\end{proposition}

\begin{proof}
We know that for all $x\in \mathbb{R}^3,$ 
$f(x)=0$ or $|f(x)|> R,$ and so for all $x\in\mathbb{R}^3,$
such that $f(x)\neq 0,$
\begin{equation}
    \frac{|f(x)|^{q-p}}{R^{q-p}}\geq 1.
\end{equation}
Therefore we may compute that
\begin{align}
    \|f\|_{L^p}^p
    &=
    \int_{\mathbb{R}^3}|f(x)|^p \diff x\\
    &\leq
    \int_{\mathbb{R}^3}|f(x)|^p 
    \frac{|f(x)|^{q-p}}{R^{q-p}} \diff x\\
    &= \frac{1}{R^{q-p}}
    \int_{\mathbb{R}^3}|f(x)|^q \diff x \\
    &=
    \frac{1}{R^{q-p}} \|f\|_{L^q}^q,
    \end{align}
and this completes the proof.
\end{proof}

\begin{theorem} \label{WeakConvergeZero}
Suppose $u\in C\left(\left[0,T_{max}\right);
H^1_{df}\right)$ is a smooth solution of the Navier--Stokes equation with $T_{max}<+\infty$.
Suppose $h\in L^1\left(\left[0,T_{max}\right);
\mathbb{R}^+\right)$, with
\begin{equation}
    \lim_{t\to T_{max}}h(t)=+\infty
\end{equation}
Let
\begin{equation}
    f(x,t)= \begin{cases}
        \lambda_2^+(x,t), &\text{if } \lambda_2^+(x,t)> h(t)
        \\
        0, &\text{if } \lambda_2^+(x,t)\leq h(t)
            \end{cases}.
\end{equation}
If 
\begin{equation}
    \limsup_{t \to T_{max}}
    \|f(\cdot,t)\|_{L^\frac{3}{2}}<+\infty,
\end{equation}
then for all $1\leq q<\frac{3}{2},$
\begin{equation}
    \lim_{t \to T_{max}}\|f(\cdot,t)\|_{L^q}=0,
\end{equation}
and 
$f(\cdot,t) \rightharpoonup 0$ weakly in $L^\frac{3}{2}$ 
as $t \to T_{max}.$
\end{theorem}

\begin{proof}
We know by hypothesis that
\begin{equation}
    \limsup_{t \to T_{max}}
    \|f(\cdot,t)\|_{L^\frac{3}{2}}<+\infty,
\end{equation}
so let 
\begin{align}
    M
    &=
    \sup_{0\leq t\leq T_{max}}
    \|f(\cdot,t)\|_{L^\frac{3}{2}}  \\
    &<
    +\infty
\end{align}
Applying Proposition \ref{LqProp}, we can immediately see that for all $1\leq q<\frac{3}{2},$
\begin{align}
    \|f(\cdot,t)\|_{L^q}^q
    &\leq
    \frac{1}{h(t)^{\frac{3}{2}-q}} 
    \|f(\cdot,t)\|_{L^\frac{3}{2}}^\frac{3}{2} \\
    &\leq
    \frac{M^\frac{3}{2}}{h(t)^{\frac{3}{2}-q}},
\end{align}
and therefore
\begin{equation}
    \|f(\cdot,t)\|_{L^q}
    \leq
    \frac{M^\frac{3}{2q}}
    {h(t)^{\frac{3}{2q}-1}}
\end{equation}
We also know that 
\begin{equation}
    \lim_{t \to T_{max}} h(t)=+\infty,
\end{equation}
so we can compute that for all $1\leq q<\frac{3}{2}$
\begin{equation}
    \lim_{t \to T_{max}}\|f(\cdot,t)\|_{L^q}=0,
\end{equation}
and this concludes the first part of the proof.

We will now use this fact to show that
$f(\cdot,t) \rightharpoonup 0$ weakly in $L^\frac{3}{2}$ 
as $t \to T_{max}.$
$L^\frac{3}{2}$ is the dual space of $L^3,$ so we will show that for all $g\in L^3,$
\begin{equation}
    \lim_{t\to T_{max}} \left<f(\cdot,t),g\right>=0.
\end{equation}
Fix $g\in L^3.$ $L^4$ is dense in $L^3$ so for all $\epsilon>0,$ there exists $w\in L^4$ such that
\begin{equation}
    \|g-w\|_{L^3}\leq \frac{\epsilon}{M}.
\end{equation}
Applying H\"older's inequality we find that
\begin{align}
    \left|\left<f(\cdot,t),g\right>\right|
    &\leq
    \left|\left<f(\cdot,t),w\right>\right|
    +\left|\left<f(\cdot,t),g-w\right>\right|\\
    &\leq
    \|f(\cdot,t)\|_{L^\frac{4}{3}}\|w\|_{L^4}
    +\|f(\cdot,t)\|_{L^\frac{3}{2}}\|g-w\|_{L^3}\\
    &\leq
    \|f(\cdot,t)\|_{L^\frac{4}{3}}\|w\|_{L^4} +\epsilon.
\end{align}
However, we have already shown that 
\begin{equation}
    \lim_{t\to T_{max}}\|f(\cdot,t)\|_{L^\frac{4}{3}}=0,
\end{equation}
so we can conclude that
\begin{equation}
    \limsup_{t\to T_{max}}
    \left|\left<f(\cdot,t),g\right>\right|
    <\epsilon.
\end{equation}
$\epsilon>0$ was arbitrary, so taking the 
limit $\epsilon\to 0,$ we find that
\begin{equation}
    \lim_{t\to T_{max}}
    \left<f(\cdot,t),g\right>=0.
\end{equation}
Therefore, for all $g\in L^3,$
\begin{equation}
    \lim_{t\to T_{max}}
    \left<f(\cdot,t),g\right>=0,
\end{equation}
and we can conclude that 
$f(\cdot,t) \rightharpoonup 0$ weakly in $L^\frac{3}{2}$ 
as $t \to T_{max}.$
This completes the proof.
\end{proof}

When Escauriaza, Seregin, and \v{S}ver\'ak \cite{ESS} proved
that if $T_{max}<+\infty,$ then 
\begin{equation}
    \limsup_{t \to T_{max}}
    \|u(\cdot,t)\|_{L^3}
    =+\infty,
\end{equation}
their proof relied on showing that 
if $T_{max}<+\infty,$ and
\begin{equation}
    \limsup_{t \to T_{max}}
    \|u(\cdot,t)\|_{L^3}
    <+\infty,
\end{equation}
then $u(\cdot,t) \rightharpoonup 0$ weakly in $L^3$
as $t \to T_{max},$ and using a backward uniqueness result to derive a contradiction.
There is no comparable backward uniqueness result for $\lambda_2^+,$ and certainly not for $f$---which is $\lambda_2^+$ restricted to points in its domain where it takes large values---so Theorem \ref{WeakConvergeZero} does not imply the endpoint regularity criterion with $\lambda_2^+\in L^\infty_T L^\frac{3}{2}_x$, but it does suggest a possible direction towards establishing the endpoint case and showing that Conjecture \ref{EigenEndpointConjecture} holds.

\section{Mixed Lebesgue Sum Spaces} \label{SumSpaceSection}

In this section, we will discuss the structure of mixed Lebesgue sum spaces in some scaling class, proving the inclusion in Theorem \ref{MixedLpSumIntro}.
We will begin by introducing the distribution function, 
which describes how the range of a function $f$ is distributed by considering the Lebesgue measure of the set of $\left\{x\in \mathbb{R}^3: |f(x)|>\alpha\right\}$.

\begin{definition}
Let $f:\mathbb{R}^3 \to \mathbb{R}^3,$ be a Lebesgue measurable function.
Then for all $\alpha\geq 0,$ let
\begin{equation}
    \lambda_f(\alpha)
    =
    \left|\left\{x\in\mathbb{R}^3:
    |f(x)|>\alpha\right\}\right|.
\end{equation}
\end{definition}

We will note that two functions with the same distribution function will have the same $L^p$ norm for all 
$1\leq p \leq +\infty,$ and for 
$1\leq p<+\infty,$ we have the following explicit formula.

\begin{proposition} \label{DistributionLp}
For all $1\leq p<+\infty$, and for all $f\in L^p\left(\mathbb{R}^3\right)$
\begin{align}
    \|f\|_{L^p}^p
    &=
    \int_{\mathbb{R}^3}|f(x)|^p \diff x \\
    &=
    p\int_0^\infty \alpha^{p-1}\lambda_f(\alpha) \diff\alpha.
\end{align}
\end{proposition}

For a proof of this result, and a good overview of the related literature, see Chapter 6.4 in \cite{Folland}.
We will also use the distribution function to define the endpoint Lorentz space
$L^{q,\infty}$, also known as weak $L^q$.

\begin{definition}
For all $1\leq q< \infty,$ 
and for all Lebesgue measurable functions
$f:\mathbb{R}^3 \to \mathbb{R}^3,$, 
\begin{equation}
    \|f\|_{L^{q,\infty}}
    =
    \left(\sup_{\alpha>0} \alpha^q \lambda_f(\alpha)
    \right)^{\frac{1}{q}}.
\end{equation}
Furthermore define $L^{q,\infty}\left(\mathbb{R}^3\right)$
by
\begin{equation}
    L^{q,\infty}\left(\mathbb{R}^3\right)
    =
    \left\{f:\mathbb{R}^3 \to \mathbb{R}^3,
    \text{Lebesgue measurable}:
    \|f\|_{L^{q,\infty}}<+\infty \right\}
\end{equation}
Note that $\|\cdot\|_{L^{q,\infty}}$ is a norm for $1<q<+\infty$, but is only a quasi-norm for $q=1,$ as the triangle inequality fails to hold.
\end{definition}

\begin{definition}
For all $1\leq p,q<+\infty,$
we will define
\begin{equation}
    L^p_T L^{q,\infty}_x
    =
    L^p\left([0,T);
    L^{q,\infty}\left(\mathbb{R}^3\right)\right),
\end{equation}
and for all $f\in L^p_T L^{q,\infty}_x$,
\begin{equation}
    \|f\|_{L^p_T L^{q,\infty}_x}
    =
    \left( \int_0^T
    \|f(\cdot,t)\|_{L^{q,\infty}}^p \diff t
    \right)^\frac{1}{p}.
\end{equation}
We will note again that this is a norm for $1<q<+\infty,$ and a quasi-norm for $q=1.$
\end{definition}

We will now prove a proposition that will be essential in the decomposition necessary to prove our sum space result.

\begin{proposition} \label{WeakLqProp}
Suppose $f\in L^{q,\infty}\left(\mathbb{R}^3\right), 1<q<+\infty$ and there exists $R>0,$ such that for all $x\in \mathbb{R}^3,$ 
$f(x)=0$ or $|f(x)|>R.$ Then for all $p<q,$
\begin{equation}
    \|f\|_{L^p}^p \leq\frac{q}{q-p} \frac{1}{R^{q-p}}
    \|f\|_{L^{q,\infty}}^q.
\end{equation}
\end{proposition}

\begin{proof}
First we will recall that
\begin{equation}
    \|f\|_{L^{q,\infty}}^q=
    \sup_{\alpha>0} \alpha^q \lambda_f(\alpha),
\end{equation}
and so therefore, for all $\alpha>0,$
\begin{equation} \label{DistrubtionControl}
    \lambda_f(\alpha)\leq \|f\|_{L^{q,\infty}}^q \alpha^{-q} 
\end{equation}
Recalling that by hypothesis
if $|f(x)|>0,$ then $|f(x)|>R,$ 
we may conclude that for all $0\leq \alpha \leq R,$
\begin{equation}
    \left\{x\in\mathbb{R}^3: |f(x)|>\alpha\right\}
    =
    \left\{x\in\mathbb{R}^3: |f(x)|>R\right\}.
\end{equation}
This implies that for all $0\leq\alpha\leq R,$
\begin{equation}
    \lambda_f(\alpha)=\lambda_f(R).
\end{equation}

We will now use Proposition \ref{DistributionLp} to estimate the $L^p$ norm, breaking up the integral into the intervals $[0,R]$ and $(R,+\infty)$:
\begin{align}
    \|f\|_{L^p}^p
    &=
    p\int_0^\infty \alpha^{p-1}\lambda_f(\alpha) \diff\alpha \\
    &=
    p\int_0^R \alpha^{p-1}\lambda_f(\alpha) \diff\alpha
    +p\int_R^\infty \alpha^{p-1}\lambda_f(\alpha) \diff\alpha \\
    &=
    \lambda_f(R)\int_0^R p \alpha^{p-1} \diff\alpha
    +p\int_R^\infty \alpha^{p-1}\lambda_f(\alpha) \diff\alpha \\
    &=
    \lambda_f(R)R^p
    +p\int_R^\infty \alpha^{p-1}\lambda_f(\alpha) \diff\alpha.
\end{align}
Applying the estimate \eqref{DistrubtionControl}, we can compute that
\begin{align}
    \|f\|_{L^p}^p
    &\leq
    \|f\|_{L^{q,\infty}}^q R^{p-q}
    +p \|f\|_{L^{q,\infty}}^q 
    \int_R^\infty \alpha^{p-q-1} \diff\alpha \\
    &=
    \|f\|_{L^{q,\infty}}^q R^{p-q}
    +\frac{p}{p-q} \|f\|_{L^{q,\infty}}^q
    \alpha^{p-q} \bigg|_{\alpha=R}^\infty \\
    &=
    \|f\|_{L^{q,\infty}}^q R^{p-q}
    +\frac{p}{q-p} \|f\|_{L^{q,\infty}}^q R^{p-q} \\
    &=
    \left(1+\frac{p}{q-p}\right)R^{p-q}
    \|f\|_{L^{q,\infty}}^q \\
    &=
    \frac{q}{q-p} \frac{1}{R^{q-p}} \|f\|_{L^{q,\infty}}^q.
\end{align}
This completes the proof.
\end{proof}

We will now prove that the continuous embedding
    $L^{q,\infty}\left(\mathbb{R}^3\right)\hookrightarrow
    L^p\left(\mathbb{R}^3\right)
    +L^\infty\left(\mathbb{R}^3\right)$
holds, along the associated inequality.
While this is a relatively standard result, we include the proof for the sake of completeness and clarity, because the proof of Theorem \ref{MixedLpSumIntro} is an adaptation of the proof of this embedding.

\begin{theorem} \label{LpSum}
Suppose $1\leq p <q <+\infty,$
then
\begin{equation}
    L^{q,\infty}\left(\mathbb{R}^3\right)\hookrightarrow
    L^p\left(\mathbb{R}^3\right)
    +L^\infty\left(\mathbb{R}^3\right),
\end{equation}
and in particular for all 
$f\in L^{q,\infty}\left(\mathbb{R}^3\right)$
\begin{equation}
    \|f\|_{L^p+L^\infty}
    \leq C_{p,q}
    \|f\|_{L^{q,\infty}},
\end{equation}
where
\begin{equation}
    C_{p,q}=\inf_{k>0} \left( k
    +\left(\frac{q}{q-p}\right)^\frac{1}{p}
    k^{1-\frac{q}{p}}\right).
\end{equation}
\end{theorem}

\begin{proof}
For all $R>0,$ let
\begin{equation}
    g_R(x,t)=
    \begin{cases}
    f(x,t), &\text{if } |f(x,t)| > R \\
    0, &\text{if } |f(x,t)| \leq R
    \end{cases}
\end{equation}
and
\begin{equation}
    h_R(x,t)=
    \begin{cases}
    f(x,t), &\text{if } |f(x,t)| \leq R \\
    0, &\text{if } |f(x,t)| > R
    \end{cases}.
\end{equation}
We can clearly see that 
\begin{equation}
    \|h_R\|_{L^\infty}\leq R,
\end{equation}
and applying Proposition \ref{WeakLqProp}, we can conclude that
\begin{equation}
    \|g_R\|_{L^p}= \left(\frac{q}{q-p}\right)^\frac{1}{p}
    R^{1-\frac{q}{p}}\|f\|_{L^{q,\infty}}^\frac{q}{p}
\end{equation}
Observing that $f=g_R+h_R,$ we can conclude that for all $R>0$
\begin{align}
    \|f\|_{L^p+L^\infty}
    &\leq 
    \|g_R\|_{L^p}+\|h_R\|_{L^\infty} \\
    &\leq
    R+\left(\frac{q}{q-p}\right)^\frac{1}{p}
    R^{1-\frac{q}{p}}\|f\|_{L^{q,\infty}}^\frac{q}{p}.
\end{align}
Now we will let $R=k \|f\|_{L^{q,\infty}},$
and observe that for all $k>0,$
\begin{equation}
    \|f\|_{L^p+L^\infty} \leq
    \left(k+\left(\frac{q}{q-p}\right)^\frac{1}{p}
    k^{1-\frac{q}{p}}\right)\|f\|_{L^{q,\infty}}.
\end{equation}
Therefore we may conclude that
\begin{equation}
    \|f\|_{L^p+L^\infty} \leq \inf_{k>0}
    \left(k+\left(\frac{q}{q-p}\right)^\frac{1}{p}
    k^{1-\frac{q}{p}}\right)\|f\|_{L^{q,\infty}},
\end{equation}
and this completes the proof.
\end{proof}

We will now prove Theorem \ref{MixedLpSumIntro}, which is restated here for the reader's convenience.

\begin{theorem} \label{MixedLpSum}
Suppose $1\leq k<+\infty, 1 \leq m<+\infty,$ and suppose
\begin{equation}
    \frac{k}{p}+\frac{m}{q}=1,
\end{equation}
and
\begin{equation}
    \frac{k}{p'}+\frac{m}{q'}=1,
\end{equation}
with $m< q' < q< +\infty.$
Then 
\begin{equation}
    L_T^p L_x^{q,\infty} \subset L_T^{p'} L_x^{q'}
    +L_T^k L_x^\infty.
\end{equation}
In particular,
for all $f\in  L_T^p L_x^{q,\infty}$,
we have the explicit decomposition,
$f=g+h$ with 
$g \in L_T^{p'} L_x^{q'}, h \in L_T^k L_x^\infty,$ where
\begin{equation}
    g(x,t)= \begin{cases}
        f(x,t), &\text{if } |f(x,t)|>
        \|f(\cdot,t)\|_{L^{q,\infty}}^{\frac{p}{k}}
        \\
        0, &\text{if } |f(x,t)|\leq
        \|f(\cdot,t)\|_{L^{q,\infty}}^{\frac{p}{k}}
            \end{cases},
\end{equation}
and
\begin{equation}
    h(x,t)= \begin{cases}
        f(x,t), &\text{if } |f(x,t)| \leq
        \|f(\cdot,t)\|_{L^{q,\infty}}^{\frac{p}{k}}
        \\
        0, &\text{if } |f(x,t)| >
        \|f(\cdot,t)\|_{L^{q,\infty}}^{\frac{p}{k}}
            \end{cases},
\end{equation}
and we have the bounds
\begin{equation} \label{BoundA}
    \int_0^T \|g(\cdot,t)\|_{L^{q'}}^{p'} \diff t
    \leq \left(\frac{q}{q-q'}\right)^\frac{p'}{q'} 
    \int_0^T \|f(\cdot,t)\|_{L^{q,\infty}}^p \diff t,
\end{equation}
and
\begin{equation} \label{BoundB}
    \int_0^T \|h(\cdot,t)\|_{L^\infty}^k \diff t
    \leq \int_0^T 
    \|f(\cdot,t)\|_{L^{q,\infty}}^p \diff t.
\end{equation}
\end{theorem}

\begin{proof}
It is immediately clear that for all $f\in L_T^p L_x^{q,\infty}, f=g+h,$ so it suffices to prove the bounds \eqref{BoundA} and \eqref{BoundB}, 
which in turn establish that
$g \in L_T^{p'} L_x^{q'}$ and $h \in L_T^k L_x^\infty$.
First we will prove the bound \eqref{BoundB}.
It is clear from the definition of $h,$
that for all $0\leq t< T,$
\begin{equation}
    \|h(\cdot,t)\|_{L^\infty} \leq
    \|f(\cdot,t)\|_{L^{q,\infty}}^{\frac{p}{k}},
\end{equation}
and therefore we may conclude that
\begin{equation}
    \int_0^T \|h(\cdot,t)\|_{L^\infty}^k \diff t
    \leq \int_0^T 
    \|f(\cdot,t)\|_{L^{q,\infty}}^p \diff t.
\end{equation}

Now we will prove the bound \eqref{BoundA}.
Letting $R=\|f(\cdot,t)\|_{L^{q,\infty}}^{\frac{p}{k}},$
we can apply Proposition \ref{WeakLqProp} and compute that for all $0\leq t < T,$
\begin{align}
    \|g(\cdot,t)\|_{L^{q'}}^{q'} 
    &\leq 
    \frac{q}{q-q'} R^{q'-q} 
    \|g(\cdot,t)\|_{L^{q,\infty}}^q \\
    &\leq 
    \frac{q}{q-q'} R^{q'-q} 
    \|f(\cdot,t)\|_{L^{q,\infty}}^q \\
    &=
    \frac{q}{q-q'}
    \|f(\cdot,t)\|_{L^{q,\infty}}^{q+\frac{p}{k}(q'-q)}.
\end{align}
Taking both sides to the power of $\frac{p'}{q'},$ 
we find that
\begin{align}
    \|g(\cdot,t)\|_{L^{q'}}^{p'}
    &\leq 
    \left(\frac{q}{q-q'}\right)^\frac{p'}{q'}
    \|f(\cdot,t)\|_{L^{q,\infty}}^{p'
    \left(\frac{q}{q'}+\frac{p}{k}
    \left(1-\frac{q}{q'}\right)\right)}\\
    &\leq \label{ScalingFamilyStep}
    \left(\frac{q}{q-q'}\right)^\frac{p'}{q'}
    \|f(\cdot,t)\|_{L^{q,\infty}}^{p \frac{p'}{p}
    \left(\frac{q}{q'}+\frac{p}{k}
    \left(1-\frac{q}{q'}\right)\right)}.
\end{align}
It remains to show that
\begin{equation}
    \frac{p'}{p}\left(\frac{q}{q'}+\frac{p}{k}
    \left(1-\frac{q}{q'}\right)\right)
    =1,
\end{equation}
which we will do now.
First we will note that
\begin{align}
    \frac{q}{q'}
    &=
    \frac{m/q'}{m/q}\\
    &=
    \frac{1-\frac{k}{p'}}{1-\frac{k}{p}}.
\end{align}
Therefore we can see that
\begin{align}
    1-\frac{q}{q'}
    &=
    1-\frac{1-\frac{k}{p'}}{1-\frac{k}{p}} \\
    &=
    \frac{\left(1-\frac{k}{p}\right)
    -\left(1-\frac{k}{p'}\right)}
    {1-\frac{k}{p}} \\
    &=
    \frac{\frac{k}{p'}-\frac{k}{p}}
    {1-\frac{k}{p}}
\end{align}
Next we can compute that
\begin{equation}
    \frac{p}{k}\left(1-\frac{q}{q'}\right)
    =
    \frac{\frac{p}{p'}-1}{1-\frac{k}{p}}
\end{equation}
And therefore we may conclude that
\begin{align}
    \frac{q}{q'}+\frac{p}{k} \left(1-\frac{q}{q'}\right)
    &=
    \frac{1-\frac{k}{p'}}{1-\frac{k}{p}}
    +\frac{\frac{p}{p'}-1}{1-\frac{k}{p}} \\
    &=
    \frac{\frac{p}{p'}-\frac{k}{p'}}{1-\frac{k}{p}}.
\end{align}
Finally multiplying by $\frac{p'}{p}$ we find that
\begin{align}
    \frac{p'}{p}\left(\frac{q}{q'}+\frac{p}{k}
    \left(1-\frac{q}{q'}\right)\right)
    &=
    \frac{p-k}{p-k}\\
    &=1.
\end{align}
Plugging this back into \eqref{ScalingFamilyStep},
we find that for all $0\leq t<T,$
\begin{equation}
    \|g(\cdot,t)\|_{L^{q'}}^{p'}
    \leq 
    \left(\frac{q}{q-q'}\right)^\frac{p'}{q'}
    \|f(\cdot,t)\|_{L^{q,\infty}}^p,
\end{equation}
and therefore that
\begin{equation}
    \int_0^T \|g(\cdot,t)\|_{L^{q'}}^{p'} \diff t
    \leq \left(\frac{q}{q-q'}\right)^\frac{p'}{q'} 
    \int_0^T \|f(\cdot,t)\|_{L^{q,\infty}}^p \diff t.
\end{equation}
This completes the proof.
\end{proof}

\begin{remark}
We will note that the proof of Theorem \ref{MixedLpSum} is an adaptation of the proof of Theorem \ref{LpSum}, with the correct choice of $R(t)$ for each time. 
When decomposing $f(\cdot,t)$ into functions in $L^{q'}_x$ and $L^\infty_x,$ it is clear that we will control the small values of $f(\cdot,t)$ in $L^\infty_x$ and the large values of $f(\cdot,t)$ in $L^{q'}_x$. The only question is the value of the cutoff function $R(t).$ In the proof of Theorem \ref{MixedLpSum}, we took 
$R(t)=\|f(\cdot,t)\|_{L^{q,\infty}}^{\frac{p}{k}}$.
While this is not the only choice of $R(t)$ available, it is clear that for any choice of $R(t)$ we will have
$R(t) \sim \|f(\cdot,t)\|_{L^{q,\infty}}^{\frac{p}{k}}$,
as the exponent is determined by scaling so any alternate choice will differ at most like a scalar multiple.

We can see that the proof requires that 
$\|f(\cdot,t)\|_{L^{q,\infty}} \in L^p\left([0,T)\right)$,
so it does not appear this condition can be weakened to weak $L^p$ in time, the condition
$\|f(\cdot,t)\|_{L^{q,\infty}} \in
L^{p,\infty}\left([0,T)\right)$.
This is fairly compelling evidence that
Conjecture \ref{MixedLpConjecture} holds and
\begin{equation}
    L_T^{p,\infty} L_x^{q,\infty} \not\subset 
    L_T^{p'} L_x^{q'}+L_T^k L_x^\infty.
\end{equation}
\end{remark}

\begin{corollary} \label{MixedLpCor}
Suppose $1\leq k<+\infty, 1 \leq m<+\infty,$ and suppose
\begin{equation}
    \frac{k}{p}+\frac{m}{q}=1,
\end{equation}
and
\begin{equation}
    \frac{k}{p'}+\frac{m}{q'}=1,
\end{equation}
with $m< q' < q< +\infty.$
Then 
\begin{equation}
    L_T^p L_x^{q,\infty} +L^k_T L^\infty_x 
    \subset 
    L_T^{p'} L_x^{q'} +L_T^k L_x^\infty.
\end{equation}
\end{corollary}

\begin{proof}
Suppose $f \in L_T^p L_x^{q,\infty} +L^k_T L^\infty_x$.
Let $f=g+h, g\in L_T^p L_x^{q,\infty},
h \in L^k_T L^\infty_x$.
Applying Theorem \ref{MixedLpSum}, we can see that
$g \in L^{p'}_T L^{q'}_x +L^k_T L^\infty_x$,
and so $g=\phi +\psi,
\phi \in L^{p'}_T L^{q'}_x, \psi \in L^k_T L^\infty_x$.
From this we may conclude that 
\begin{equation}
    f=\phi+\psi+h,
\end{equation}
with $\phi \in L^{p'}_T L^{q'}_x, 
\psi+h\in L^k_T L^\infty_x$.
Therefore we may conclude that for all 
$f\in L_T^p L_x^{q,\infty} +L^k_T L^\infty_x,$ we have
$f\in L_T^{p'} L_x^{q'} +L_T^k L_x^\infty$,
and this completes the proof.
\end{proof}

\begin{remark}
While the results in this section, particularly Theorem \ref{MixedLpSum}, were proven in terms of the Lebesgue measure on $\mathbb{R}^3,$ we did not use any of the specific properties of the Lebesgue measure on $\mathbb{R}^3$ in the proof, and these results will in fact hold for any Borel measure on a measure space.
\end{remark}

Before moving on to regularity criteria in sum spaces in terms of the velocity, we will use Theorem \ref{MixedLpSum} and Corollary \ref{MixedLpCor} to strengthen the regularity criterion on $\lambda_2^+$ from the space
$L^p_T L^q_x +L^1_T L^\infty_x$ 
to the slightly larger space
$L^p_T L^{q,\infty}_x +L^1_T L^\infty_x$,
proving Corollary \ref{StrainSumSpaceWeakIntro}, which is restated here for the reader's convenience.

\begin{corollary} \label{StrainSumSpaceWeak}
Suppose $u\in C\left(\left[0,T_{max}\right);
\dot{H}^1_{df}\right)$ is a smooth solution of the Navier--Stokes equation. 
Let $\frac{3}{2}<q<+\infty, \frac{2}{p}+\frac{3}{q}=2,$ 
and let 
$\lambda_2^+ =f+g.$ 
Then for all $0<T<T_{max}$
\begin{equation} \label{EnstrophyBoundLambdaWeak}
    \|S(\cdot,T)\|_{L^2}^2 \leq 
    \left\| S^0 \right\|_{L^2}^2
    \exp\left( \Tilde{C}_p
    \int_0^T \|f(\cdot,t)\|_{L^{q,\infty}}^p \diff t
    +2\int_0^T \|g(\cdot,t)\|_{L^\infty} \diff t
    \right),
\end{equation}
where
\begin{equation}
    \Tilde{C}_p=
    \frac{C_{p'}}{\nu^{p'-1}}
    \left(\frac{q}{q-q'}\right)^\frac{p'}{q'} +2,
\end{equation}
with $C_{p'}$ taken as in Theorem \ref{StrainSumSpace},
and $\frac{3}{2}<q'<q, \frac{2}{p'}+\frac{3}{q'}=2.$
In particular if $T_{max}<+\infty,$ then
\begin{equation}
    \Tilde{C}_p \int_0^T 
    \|f(\cdot,t)\|_{L^{q,\infty}}^p \diff t
    +2\int_0^T \|g(\cdot,t)\|_{L^\infty} \diff t
    =+\infty.
\end{equation}
\end{corollary}

\begin{proof}
We know that if $T_{max}<+\infty,$ then 
\begin{equation}
    \lim_{T \to T_{max}}\|S(\cdot,T)\|_{L^2}^2=+\infty,    
\end{equation}
so it suffices to prove the bound
\eqref{EnstrophyBoundLambdaWeak}.

We will begin by fixing $\frac{3}{2}<q'<q$ and setting
\begin{equation}
    h(x,t)= \begin{cases}
        f(x,t), &\text{if } |f(x,t)|>
        \|f(\cdot,t)\|_{L^{q,\infty}}^p
        \\
        0, &\text{if } |f(x,t)|\leq
        \|f(\cdot,t)\|_{L^{q,\infty}}^p
            \end{cases},
\end{equation}
and
\begin{equation}
    \psi(x,t)= \begin{cases}
        f(x,t), &\text{if } |f(x,t)| \leq
        \|f(\cdot,t)\|_{L^{q,\infty}}^p
        \\
        0, &\text{if } |f(x,t)| >
        \|f(\cdot,t)\|_{L^{q,\infty}}^p
            \end{cases}.
\end{equation}
It is clear that $f=h+\psi$,
and applying Theorem \ref{MixedLpSum},
we have the bounds
\begin{equation} \label{EigenBoundA}
    \int_0^T \|h(\cdot,t)\|_{L^{q'}}^{p'} \diff t
    \leq \left(\frac{q}{q-q'}\right)^\frac{p'}{q'} 
    \int_0^T \|f(\cdot,t)\|_{L^{q,\infty}}^p \diff t,
\end{equation}
and
\begin{equation} \label{EigenBoundB}
    \int_0^T \|\psi(\cdot,t)\|_{L^\infty} \diff t
    \leq \int_0^T 
    \|f(\cdot,t)\|_{L^{q,\infty}}^p \diff t.
\end{equation}
Recalling that $\lambda_2^+=h+\psi+g,$ and applying \eqref{EigenBoundA}, \eqref{EigenBoundB}, and Theorem \ref{StrainSumSpace}, we can conclude that
for all $\frac{3}{2}<q'<q,$
and for all $0<T<T_{max}$
\begin{align}
    \|S(\cdot,T)\|_{L^2}^2 
    &\leq 
    \left\| S^0 \right\|_{L^2}^2
    \exp\left(\frac{C_{p'}}{\nu^{p'-1}} \int_0^T 
    \|h(\cdot,t)\|_{L^{q'}}^{p'} \diff t
    +2\int_0^T \|\psi(\cdot,t)+g(\cdot,t)\|_{L^\infty} 
    \diff t \right) \\
    &\leq
    \left\| S^0 \right\|_{L^2}^2
    \exp\left(\left(\frac{C_{p'}}{\nu^{p'-1}}
    \left(\frac{q}{q-q'}\right)^\frac{p'}{q'} +2\right)
    \int_0^T \|f(\cdot,t)\|_{L^{q,\infty}}^p \diff t
    +2\int_0^T \|g(\cdot,t)\|_{L^\infty} \diff t
    \right).
\end{align}
This completes the proof.

\end{proof}

\section{Velocity regularity criterion} \label{VelocitySection}

In this section, we will consider regularity criteria for the Navier--Stokes equation in sum spaces in terms of the velocity.
We will begin by proving Theorem \ref{VelocitySumSpaceIntro}, which is restated here for the reader's convenience.

\begin{theorem} \label{VelocitySumSpace}
Suppose $u\in C\left(\left[0,T_{max}\right);
\dot{H}^1_{df}\right)$ is a smooth solution of the Navier--Stokes equation. Let $3<q<+\infty,
\frac{2}{p}+\frac{3}{q}=1,$ and let 
$ u=v+\sigma.$ Then for all $0<T<T_{max}$
\begin{equation} \label{EnstrophyBoundVelocity}
    \|\nabla u(\cdot,T)\|_{L^2}^2 \leq 
    \left\| \nabla u^0 \right\|_{L^2}^2
    \exp\left(\frac{C_p}{\nu^{p-1}} \int_0^T 
    \|v(\cdot,t)\|_{L^q}^p \diff t
    +\frac{1}{\nu} \int_0^T \|\sigma(\cdot,t)\|_{L^\infty}^2
    \diff t \right),
\end{equation}
where $C_p$ depends only on $p.$
In particular if $T_{max}<+\infty,$ then
\begin{equation}
    \frac{C_p}{\nu^{p-1}} \int_0^{T_{max}} 
    \|v(\cdot,t)\|_{L^q}^p \diff t
    +\frac{1}{\nu} \int_0^{T_{max}} 
    \|\sigma(\cdot,t)\|_{L^\infty}^2 \diff t
    =+\infty.
\end{equation}
\end{theorem}

\begin{proof}
We will begin by observing that if $T_{max}<+\infty,$ then
\begin{equation}
    \lim_{T \to T_{max}}\|\nabla u(\cdot,T)\|_{L^2}^2= +\infty,
\end{equation}
so it suffices to prove the bound \eqref{EnstrophyBoundVelocity}.
To prove this bound we will make use of our bound for enstrophy growth in terms of the velocity, computing that
\begin{align}
\partial_t \|\nabla u(\cdot,t)\|_{L^2}^2
&=
-2\nu\|-\Delta u\|_{L^2}^2
-2\left<(u\cdot \nabla)u, -\Delta u\right>\\
&=
-2\nu\|-\Delta u\|_{L^2}^2
-2\left<(v\cdot \nabla)u, -\Delta u\right>
-2\left<(\sigma\cdot \nabla)u, -\Delta u\right>\\
&\leq \label{Step1}
-2\nu\|-\Delta u\|_{L^2}^2
+2\|v\|_{L^q}\|\nabla u\|_{L^r}\|-\Delta u\|_{L^2}
+2\|\sigma\|_{L^\infty} \|\nabla u\|_{L^2}\|-\Delta u\|_{L^2},
\end{align}
where $\frac{1}{q}+\frac{1}{r}=\frac{1}{2},$ and we have applied H\"older's inequality.
Applying Young's inequality we find that
\begin{equation} \label{Step2}
    2\|\sigma\|_{L^\infty} \|\nabla u\|_{L^2}\|-\Delta u\|_{L^2}
    \leq \frac{1}{\nu}
    \|\sigma\|_{L^\infty}^2\|\nabla u\|_{L^2}^2
    +\nu \|-\Delta u\|_{L^2}^2.
\end{equation}

Now we need to bound the term $2\|v\|_{L^q}\|\nabla u\|_{L^r}\|-\Delta u\|_{L^2}.$
Observe that $\frac{1}{r}=\frac{1}{2}-\frac{1}{q},$ 
with $3<q<+\infty,$ so we can conclude that
$2<r<6.$
Let $\rho=\frac{3}{q}$.
Observe that $0<\rho<1$ and
\begin{align}
    (1-\rho)\frac{1}{2}+\rho\frac{1}{6}
    &=
    \frac{1}{2}\left(1-\frac{3}{q}\right)
    +\frac{1}{6}\frac{3}{q}\\
    &=
    \frac{1}{2}-\frac{1}{q}\\
    &=
    \frac{1}{r}.
\end{align}
Therefore, interpolating between $L^2$ and $L^6$ we can see that
\begin{align}
    \|\nabla u\|_{L^r}
    &\leq
    \|\nabla u\|_{L^2}^{1-\frac{3}{q}}\|\nabla u\|_{L^6}^{\frac{3}{q}}\\
    &=
    \|\nabla u\|_{L^2}^{\frac{2}{p}}\|\nabla u\|_{L^6}^{1-\frac{2}{p}},
\end{align}
recalling that by hypothesis $\frac{2}{p}+\frac{3}{q}=1.$
We can apply the Sobolev inequality to conclude that
\begin{equation}
    \|\nabla u\|_{L^6} \leq C\|-\Delta u\|_{L^2},
\end{equation}
and putting this together with the interpolation inequality we find that 
\begin{equation}
    2\|v\|_{L^q}\|\nabla u\|_{L^r}\|-\Delta u\|_{L^2}
    \leq C
    \|v\|_{L^q}\|\nabla u\|_{L^2}^\frac{2}{p}
    \|-\Delta u\|_{L^2}^{2-\frac{2}{p}}.
\end{equation}
Recalling that $2<p<+\infty$, take $1<b<2,$ such that $\frac{1}{p}+\frac{1}{b}=1.$
We can then rewrite the bound in terms of $b$ and apply Young's inequality with exponents $p,b$ to find that
\begin{align}
    2\|v\|_{L^q}\|\nabla u\|_{L^r}\|-\Delta u\|_{L^2}
    &\leq C
    \|v\|_{L^q}\|\nabla u\|_{L^2}^\frac{2}{p}
    \|-\Delta u\|_{L^2}^{\frac{2}{b}}\\
    &=
    \frac{1}{\nu^{\frac{1}{b}}}\|v\|_{L^q}\|\nabla u\|_{L^2}^\frac{2}{p}
    \nu^\frac{1}{b}\|-\Delta u\|_{L^2}^{\frac{2}{b}}\\
    &\leq
    \frac{C_p}{\nu^\frac{p}{b}}\|v\|_{L^q}^p
    \|\nabla u\|_{L^2}^2+ \nu \|-\Delta u\|_{L^2}^2 \\
    &=  \label{Step3}
    \frac{C_p}{\nu^{p-1}}\|v\|_{L^q}^p
    \|\nabla u\|_{L^2}^2+ \nu \|-\Delta u\|_{L^2}^2.
\end{align}

Putting together 
\eqref{Step1},\eqref{Step2}, and \eqref{Step3},
we find that for all $0<t<T_{max,}$
\begin{equation}
    \partial_t\|\nabla u(\cdot,t)\|_{L^2}^2
    \leq
    \left(\frac{C_p}{\nu^{p-1}}\|v\|_{L^q}^p
    +\frac{1}{\nu}\|\sigma\|_{L^\infty}^2\right)
    \|\nabla u\|_{L^2}^2.
\end{equation}
Applying Gr\"onwall's inequality, we can conclude that
for all $0<T<T_{max},$
\begin{equation}
    \|\nabla u(\cdot,T)\|_{L^2}^2 \leq 
    \left\| \nabla u^0 \right\|_{L^2}^2
    \exp\left(\frac{C_p}{\nu^{p-1}} \int_0^T 
    \|v(\cdot,t)\|_{L^q}^p \diff t
    +\frac{1}{\nu} \int_0^T \|\sigma(\cdot,t)\|_{L^\infty}^2
    \diff t \right),
\end{equation}
and this completes the proof.
\end{proof}

We will now prove a corollary that requires the concentration of the $L^p_T L^q_x$ norm of $u$ at large values in the range 
for all 
$3<q<+\infty, \frac{2}{p}+\frac{3}{q}=1$,
as $t\to T_{max}$, when $T_{max}<+\infty.$

\begin{corollary} \label{VelocityCor}
Suppose $u\in C\left(\left[0,T_{max}\right);
\dot{H}^1_{df}\right)$ is a smooth solution of the Navier--Stokes equation, and
suppose $h\in L^2\left(\left[0,T_{max}\right);
\mathbb{R}^+\right)$. 
Let $3<q<+\infty,
\frac{2}{p}+\frac{3}{q}=1,$ and let
\begin{equation}
    v(x,t)= \begin{cases}
        u(x,t), &\text{if } |u(x,t)|> h(t)
        \\
        0, &\text{if } |u(x,t)|\leq h(t) 
        \end{cases}.
\end{equation}
Then for all $0<T<T_{max}$
\begin{equation}
    \|\nabla u(\cdot,T)\|_{L^2}^2 \leq 
    \left\| \nabla u^0 \right\|_{L^2}^2
    \exp\left(\frac{C_p}{\nu^{p-1}} \int_0^T 
    \|v(\cdot,t)\|_{L^q}^p \diff t
    +\frac{1}{\nu} \int_0^T h(t)^2
    \diff t \right),
\end{equation}
where $C_p$ depends only on $p.$
In particular if $T_{max}<+\infty,$ then
\begin{equation}
    \int_0^{T_{max}} 
    \|v(\cdot,t)\|_{L^q}^p \diff t
    =+\infty.
\end{equation}
\end{corollary}

\begin{proof}
We will begin by defining
\begin{equation}
    \sigma(x,t)= \begin{cases}
        u(x,t), &\text{if } |u(x,t)|\leq h(t) \\
        0, &\text{if } |u(x,t)|> h(t)
        \end{cases}.
\end{equation}
We can see immediately that for all $0<t<T_{max},$
\begin{equation}
   \|\sigma(\cdot,t)\|_{L^\infty}\leq h(t), 
\end{equation}
and that
\begin{equation}
    u=v+\sigma.
\end{equation}
Therefore we can apply Theorem \ref{VelocitySumSpace}
and find that
\begin{align}
    \|\nabla u(\cdot,T)\|_{L^2}^2 
    &\leq 
    \left\| \nabla u^0 \right\|_{L^2}^2
    \exp\left(\frac{C_p}{\nu^{p-1}} \int_0^T 
    \|v(\cdot,t)\|_{L^q}^p \diff t
    +\frac{1}{\nu}\int_0^T \|\sigma(\cdot,t)\|_{L^\infty}^2
    \diff t \right)\\
    &\leq
    \left\| \nabla u^0 \right\|_{L^2}^2
    \exp\left(\frac{C_p}{\nu^{p-1}} \int_0^T 
    \|v(\cdot,t)\|_{L^q}^p \diff t
    +\frac{1}{\nu}\int_0^T  h(t)^2 \diff t
    \right).
\end{align}

Next we will note, as in Theorem \ref{VelocitySumSpace}, that if $T_{max}<+\infty,$ then
\begin{equation}
    \lim_{T\to T_{max}}\|\nabla u(\cdot,T)\|_{L^2}^2=+\infty.
\end{equation}
Therefore we can conclude that if $T_{max}<+\infty,$ then
\begin{equation}
    \frac{C_p}{\nu^{p-1}} \int_0^{T_{max}} 
    \|v(\cdot,t)\|_{L^q}^p \diff t
    +\frac{1}{\nu}\int_0^T  h(t)^2 \diff t
    =+\infty.
\end{equation}
However, we know by hypothesis that
\begin{equation}
    \int_0^{T_{max}}  h(t)^2 \diff t<+\infty,
\end{equation}
so we may conclude that
\begin{equation}
    \int_0^{T_{max}} 
    \|v(\cdot,t)\|_{L^q}^p \diff t
    =+\infty.
\end{equation}
This completes the proof.
\end{proof}

Using Theorem \ref{VelocitySumSpace} and Theorem \ref{MixedLpSum}, we will now extend our regularity criterion for the velocity from the space 
$L^p_T L^q_x+ L^2_T L^\infty_x$
to the slightly larger space
$L^p_T L^{q,\infty}_x+ L^2_T L^\infty_x$,
proving Corollary \ref{VelocitySumSpaceWeakIntro},
which is restated here for the reader's convenience.

\begin{corollary} \label{VelocitySumSpaceWeak}
Suppose $u\in C\left(\left[0,T_{max}\right);
\dot{H}^1_{df}\right)$ is a smooth solution of the Navier--Stokes equation. Let $3<q<+\infty,
\frac{2}{p}+\frac{3}{q}=1,$ and let 
$ u=v+\sigma.$ Then for all $0<T<T_{max}$
\begin{equation} \label{EnstrophyBoundVelocityWeak}
    \|\nabla u(\cdot,T)\|_{L^2}^2 \leq 
    \left\| \nabla u^0 \right\|_{L^2}^2
    \exp\left(\Tilde{C}_p \int_0^T 
    \|v(\cdot,t)\|_{L^{q,\infty}}^p \diff t
    +\frac{2}{\nu} \int_0^T \|\sigma(\cdot,t)\|_{L^\infty}^2
    \diff t \right),
\end{equation}
where
\begin{equation}
    \Tilde{C}_p=
    \frac{C_{p'}}{\nu^{p'-1}}
    \left(\frac{q}{q-q'}\right)^\frac{p'}{q'}
    +\frac{2}{\nu},
\end{equation}
with $C_{p'}$ taken as in Theorem \ref{VelocitySumSpace},
and $3<q'<q, \frac{2}{p'}+\frac{3}{q'}=1.$
In particular if $T_{max}<+\infty,$ then
\begin{equation}
    \Tilde{C}_p \int_0^T 
    \|v(\cdot,t)\|_{L^{q,\infty}}^p \diff t
    +\frac{2}{\nu} \int_0^T \|\sigma(\cdot,t)\|_{L^\infty}^2
    \diff t 
    =+\infty.
\end{equation}
\end{corollary}

\begin{proof}
We know that if $T_{max}<+\infty,$ then 
\begin{equation}
    \lim_{T \to T_{max}}\|\nabla u(\cdot,T)\|_{L^2}^2
    =+\infty,   
\end{equation}
so it suffices to prove the bound
\eqref{EnstrophyBoundVelocityWeak}.

We will begin by fixing $3<q'<q$ and setting
\begin{equation}
    \phi(x,t)= \begin{cases}
        v(x,t), &\text{if } |v(x,t)|>
        \|v(\cdot,t)\|_{L^{q,\infty}}^\frac{p}{2}
        \\
        0, &\text{if } |v(x,t)|\leq
        \|v(\cdot,t)\|_{L^{q,\infty}}^\frac{p}{2}
            \end{cases},
\end{equation}
and
\begin{equation}
    \psi(x,t)= \begin{cases}
        v(x,t), &\text{if } |v(x,t)| \leq
        \|v(\cdot,t)\|_{L^{q,\infty}}^\frac{p}{2}
        \\
        0, &\text{if } |v(x,t)| >
        \|v(\cdot,t)\|_{L^{q,\infty}}^\frac{p}{2}
            \end{cases}.
\end{equation}
It is clear that $v=\phi+\psi$,
and applying Theorem \ref{MixedLpSum},
we have the bounds
\begin{equation} \label{VelocityBoundA}
    \int_0^T \|\phi(\cdot,t)\|_{L^{q'}}^{p'} \diff t
    \leq \left(\frac{q}{q-q'}\right)^\frac{p'}{q'} 
    \int_0^T \|v(\cdot,t)\|_{L^{q,\infty}}^p \diff t,
\end{equation}
and
\begin{equation} \label{VelocityBoundB}
    \int_0^T \|\psi(\cdot,t)\|_{L^\infty}^2 \diff t
    \leq \int_0^T 
    \|v(\cdot,t)\|_{L^{q,\infty}}^p \diff t.
\end{equation}
Recalling that $\omega=\phi+\psi+\sigma,$ and applying \eqref{VelocityBoundA}, \eqref{VelocityBoundB}, and Theorem \ref{VelocitySumSpace}, we can conclude that
for all $3<q'<q, \frac{2}{p'}+\frac{3}{q'}=1$,
and for all $0<T<T_{max}$,
\begin{align}
    \|\nabla u(\cdot,T)\|_{L^2}^2 
    &\leq 
    \left\| \nabla u^0 \right\|_{L^2}^2
    \exp\left(\frac{C_{p'}}{\nu^{p'-1}} \int_0^T 
    \|\phi(\cdot,t)\|_{L^{q'}}^{p'} \diff t
    +\frac{1}{\nu}\int_0^T \|\psi(\cdot,t)
    +\sigma(\cdot,t)\|_{L^\infty}^2 \diff t \right) \\
    &\leq
    \left\| \nabla u^0 \right\|_{L^2}^2
    \exp\left(\frac{C_{p'}}{\nu^{p'-1}} \int_0^T 
    \|\phi(\cdot,t)\|_{L^{q'}}^{p'} \diff t
    +\frac{1}{\nu}\int_0^T 
    \left(\|\psi(\cdot,t)\|_{L^\infty}
    +\|\sigma(\cdot,t)\|_{L^\infty}\right)^2 
    \diff t \right) \\
    &\leq
    \left\| \nabla u^0 \right\|_{L^2}^2
    \exp\left(\frac{C_{p'}}{\nu^{p'-1}} \int_0^T 
    \|\phi(\cdot,t)\|_{L^{q'}}^{p'} \diff t
    +\frac{2}{\nu}\int_0^T 
    \|\psi(\cdot,t)\|_{L^\infty}^2
    +\|\sigma(\cdot,t)\|_{L^\infty}^2 
    \diff t \right) \\
    &\leq
    \left\| \nabla u^0 \right\|_{L^2}^2
    \exp\left(\left(\frac{C_{p'}}{\nu^{p'-1}}
    \left(\frac{q}{q-q'}\right)^\frac{p'}{q'} 
    +\frac{2}{\nu}\right)
    \int_0^T \|v(\cdot,t)\|_{L^{q,\infty}}^p \diff t
    +\frac{2}{\nu}\int_0^T 
    \|\sigma(\cdot,t)\|_{L^\infty} \diff t
    \right).
\end{align}
This completes the proof.
\end{proof}

We will finish this section by proving the endpoint regularity criterion Theorem \ref{VelocityEndpointIntro}, which is restated here for the reader's convenience.

\begin{theorem} \label{VelocityEndpoint}
Suppose $u\in C\left(\left[0,T_{max}\right);
\dot{H}^1_{df}\right)$ is a smooth solution of the Navier--Stokes equation, and
suppose $h\in L^2\left(\left[0,T_{max}\right);
\mathbb{R}^+\right)$.
Let
\begin{equation}
    v(x,t)= \begin{cases}
        u(x,t), &\text{if } |u(x,t)|> h(t)
        \\
        0, &\text{if } |u(x,t)|\leq h(t) 
        \end{cases}.
\end{equation}
If $T_{max}<+\infty,$ then
\begin{equation}
    \limsup_{t \to T_{max}}
    \left\|v(\cdot,t)\right\|_{L^3}
    \geq \sqrt{3} \left(\frac{\pi}{2}\right)^\frac{2}{3} \nu.
\end{equation}
\end{theorem}

\begin{proof}
Suppose towards contradiction that $T_{max}<+\infty$ and 
\begin{equation}
    \limsup_{t \to T_{max}}
    \left\|v(\cdot,t)\right\|_{L^3}
    < \sqrt{3} \left(\frac{\pi}{2}\right)^\frac{2}{3} \nu.
\end{equation}
Then there exists $\epsilon,\delta>0,$ 
such that for all $T_{max}-\delta<t<T_{max},$
\begin{equation} \label{VelocityAssumption}
    \left\|v(\cdot,t)\right\|_{L^3}
    <
    \sqrt{3}\left(\frac{\pi}{2}\right)^\frac{2}{3}\nu
    -\sqrt{3}\left(\frac{\pi}{2}\right)^\frac{2}{3}\epsilon.
\end{equation}
We will again define
\begin{equation}
    \sigma(x,t)= \begin{cases}
        u(x,t), &\text{if } |u(x,t)|\leq h(t) \\
        0, &\text{if } |u(x,t)|> h(t)
        \end{cases}.
\end{equation}
We can see immediately that for all $0<t<T_{max},$
\begin{equation}
   \|\sigma(\cdot,t)\|_{L^\infty}\leq h(t), 
\end{equation}
and that
\begin{equation}
    u=v+\sigma.
\end{equation}
Now we can use our identity for enstrophy growth in terms of velocity, H\"older's inequality, and the Sobolev inequality to compute that for all $T_{max}-\delta<t<T_{max}$
\begin{align}
    \partial_t\frac{1}{2}\|\nabla u(\cdot,t)\|_{L^2}^2
    &=
    -\nu\|-\Delta u\|_{L^2}^2
    -\left<(u\cdot\nabla)u,-\Delta u\right>\\
    &=
    -\nu\|-\Delta u\|_{L^2}^2
    -\left<(v\cdot\nabla)u,-\Delta u\right>    
    -\left<(\sigma\cdot\nabla)u,-\Delta u\right>\\
    &\leq
    -\nu\|-\Delta u\|_{L^2}^2
    +\|v\|_{L^3}\|\nabla u\|_{L^6}\|-\Delta u\|_{L^2}
    +\|\sigma\|_{L^\infty}\|\nabla u\|_{L^2}
    \|-\Delta u\|_{L^2}\\
    &\leq
    -\nu\|-\Delta u\|_{L^2}^2
    +\frac{1}{\sqrt{3}}\left(\frac{2}{\pi}\right)^\frac{2}{3}
    \|v\|_{L^3}\|-\Delta u\|_{L^2}^2
    +h \|\nabla u\|_{L^2}
    \|-\Delta u\|_{L^2}.
\end{align}

We know from our hypothesis \eqref{VelocityAssumption} that
\begin{equation}
    \frac{1}{\sqrt{3}}\left(\frac{2}{\pi}\right)^\frac{2}{3}
    \|v\|_{L^3} \|-\Delta u\|_{L^2}^2
    <
    (\nu-\epsilon)\|-\Delta u\|_{L^2}^2,
\end{equation}
so we can apply Young's inequality and conclude that
\begin{align}
    \partial_t \frac{1}{2}\|\nabla u(\cdot,t)\|_{L^2}^2
    &\leq
    -\epsilon\|-\Delta u\|_{L^2}^2
    +h \|\nabla u\|_{L^2}
    \|-\Delta u\|_{L^2}\\
    &\leq
    \frac{1}{4\epsilon}h^2
    \|\nabla u\|_{L^2}^2.
\end{align}
Multiplying both sides by $2,$ we find that 
\begin{equation}
    \partial_t \|\nabla u(\cdot,t)\|_{L^2}^2
    \leq
    \frac{1}{2\epsilon} h^2 \|\nabla u\|_{L^2}^2.
\end{equation}
Applying Gr\"onwall's inequality we find that 
for all $T_{max}-\delta<T<T_{max},$
\begin{equation}
    \|\nabla u(\cdot,T)\|_{L^2}^2
    \leq \|\nabla u(\cdot,T_{max}-\delta)\|_{L^2}^2
    \exp \left(\frac{1}{2\epsilon} 
    \int_{T_{max}-\delta}^T h(t)^2 \diff t\right).
\end{equation}
Using the assumption that
$h\in L^2\left(\left[0,T_{max}\right);
\mathbb{R}^+\right),$ we can conclude that
\begin{align}
    \limsup_{T \to T_{max}} \|\nabla u(\cdot,T)\|_{L^2}^2
    &\leq
    \left\|\nabla u(\cdot,T_{max}-\delta) \right\|_{L^2}^2
    \exp\left(\frac{1}{2\epsilon} 
    \int_{T_{max}-\delta}^{T_{max}}h(t)^2 \diff t\right)\\
    &<
    +\infty.
\end{align}
This contradicts our assumption that $T_{max}<+\infty,$ so this completes the proof.
\end{proof}

\section{Vorticity regularity criterion} \label{VorticitySection}

In this section, we will consider regularity criteria for the Navier--Stokes equation in sum spaces in terms of the vorticity.
We will begin by proving Theorem \ref{VortSumSpaceIntro}, which is restated here for the reader's convenience.

\begin{theorem} \label{VortSumSpace}
Suppose $u\in C\left(\left[0,T_{max}\right);
\dot{H}^1_{df}\right)$ is a smooth solution of the Navier--Stokes equation. Let $\frac{3}{2}<q<+\infty,
\frac{2}{p}+\frac{3}{q}=2,$ and let 
$\omega =v+\sigma.$ Then for all $0<T<T_{max}$,
\begin{equation} \label{EnstrophyBoundVort}
    \|\omega(\cdot,T)\|_{L^2}^2 \leq 
    \left\| \omega^0 \right\|_{L^2}^2
    \exp\left(\frac{C_p}{\nu^{p-1}} \int_0^T 
    \|v(\cdot,t)\|_{L^q}^p \diff t
    +\sqrt{2}\int_0^T \|\sigma(\cdot,t)\|_{L^\infty} \diff t
    \right),
\end{equation}
where $C_p$ depends only on $p.$
In particular if $T_{max}<+\infty,$ then
\begin{equation}
    \frac{C_p}{\nu^{p-1}} \int_0^{T_{max}} 
    \|v(\cdot,t)\|_{L^q}^p \diff t
    +\sqrt{2}\int_0^{T_{max}} \|\sigma(\cdot,t)\|_{L^\infty} \diff t
    =+\infty.
\end{equation}
\end{theorem}

\begin{proof}
We will first observe that if $T_{max}<+\infty,$ then
\begin{equation}
    \lim_{T \to T_{max}}\|\omega(\cdot,T)\|_{L^2}
    =+\infty,
\end{equation}
and therefore it suffices to prove the bound \eqref{EnstrophyBoundVort}.
Applying our standard identity for enstrophy growth we can see that for all $0<t<T_{max}$
\begin{align}
    \partial_t \frac{1}{2} \|\omega(\cdot,t)\|_{L^2}^2
    &=
    -\nu \|\omega\|_{\dot{H}^1}^2
    +\left<S;\omega \otimes \omega\right>\\
    &=
    -\nu \|\omega\|_{\dot{H}^1}^2
    +\left<S\omega; \omega\right>\\
    &=
    -\nu \|\omega\|_{\dot{H}^1}^2
    +\left<S\omega; v\right>
    +\left<S\omega; \sigma\right>.
\end{align}
Applying H\"older's inequality with exponents 
$q,r$ and $1,\infty$ we find.
\begin{align}
    \partial_t \frac{1}{2} \|\omega(\cdot,t)\|_{L^2}^2
    &\leq 
    -\nu \|\omega\|_{\dot{H}^1}^2
    +\|S\omega\|_{L^r} \|v\|_{L^q}
    +\|S\omega\|_{L^1}\|\sigma\|_{L^\infty}\\
    &\leq 
    -\nu \|\omega\|_{\dot{H}^1}^2
    +\|S\|_{L^{2r}} \|\omega\|_{L^{2r}} \|v\|_{L^q}
    +\|S\|_{L^2}\|\omega\|_{L^2}\|\sigma\|_{L^\infty}.
\end{align}

Next we observe that $\frac{3}{2}<q<\infty,$ and so
$1<r<3,$ and consequently $2<2r<6.$ Let $\rho=\frac{3}{2q}.$
We can see that $0<\rho<1,$ and
\begin{align}
    (1-\rho) \frac{1}{2}+ \rho \frac{1}{6}
    &=
    \frac{1}{2}-\frac{\rho}{3}\\
    &=
    \frac{1}{2}-\frac{1}{2q}\\
    &=
    \frac{1}{2}-\frac{1}{2}\left(1-\frac{1}{r}\right)\\
    &=
    \frac{1}{2r}.
\end{align}
Therefore, we can interpolate between $L^2$ and $L^6$ and find that
\begin{equation}
    \|S\|_{L^{2r}}\leq 
    \|S\|_{L^2}^{1-\frac{3}{2q}} \|S\|_{L^6}^{\frac{3}{2q}}
\end{equation}
and
\begin{equation}
    \|\omega\|_{L^{2r}}\leq 
    \|\omega\|_{L^2}^{1-\frac{3}{2q}} \|\omega\|_{L^6}^{\frac{3}{2q}}
\end{equation}
Applying these interpolation inequalities, the Sobolev inequality, and the isometry in Proposition \ref{isometry},
we find that
\begin{align}
    \partial_t \frac{1}{2} \|\omega(\cdot,t)\|_{L^2}^2
    &\leq 
    -\nu \|\omega\|_{\dot{H}^1}^2
    +\|S\|_{L^2}^{1-\frac{3}{2q}}
    \|S\|_{L^6}^{\frac{3}{2q}}
    \|\omega\|_{L^2}^{1-\frac{3}{2q}} \|\omega\|_{L^6}^{\frac{3}{2q}} 
    \|v\|_{L^q}
    +\|S\|_{L^2}\|\omega\|_{L^2}\|\sigma\|_{L^\infty}\\
    &\leq
    -\nu \|\omega\|_{\dot{H}^1}^2
    +C\|S\|_{L^2}^{1-\frac{3}{2q}}
    \|S\|_{\dot{H}^1}^{\frac{3}{2q}}
    \|\omega\|_{L^2}^{1-\frac{3}{2q}} \|\omega\|_{\dot{H}^1}^{\frac{3}{2q}} 
    \|v\|_{L^q}
    +\|S\|_{L^2}\|\omega\|_{L^2}\|\sigma\|_{L^\infty} \\
    &\leq
    -\nu \|\omega\|_{\dot{H}^1}^2
    +C \|\omega\|_{L^2}^{2-\frac{3}{q}} \|\omega\|_{\dot{H}^1}^{\frac{3}{q}} 
    \|v\|_{L^q}
    +\frac{1}{\sqrt{2}}\|\omega\|_{L^2}^2 
    \|\sigma\|_{L^\infty}.
\end{align}
Multiplying both sides by $2$ and substituting
$2-\frac{3}{q}=\frac{2}{p}$ we find that
\begin{equation}
    \partial_t \|\omega(\cdot,t)\|_{L^2}^2
    =
   -2\nu \|\omega\|_{\dot{H}^1}^2
    +C \|\omega\|_{L^2}^{\frac{2}{p}} \|\omega\|_{\dot{H}^1}^{\frac{3}{q}} 
    \|v\|_{L^q}
    +\sqrt{2}\|\omega\|_{L^2}^2 \|\sigma\|_{L^\infty}.
\end{equation}

Let $b=\frac{2q}{3}.$ Clearly $1<b<+\infty,$ and recalling that 
$\frac{2}{p}+\frac{3}{q}=2,$ we can see that
\begin{align}
    \frac{1}{p}+\frac{1}{b}
    &=
    \frac{1}{p}+\frac{3}{2q}\\
    &=
    1.
\end{align}
Applying Young's inequality with exponents $p,b$ we find
\begin{equation}
    \frac{C}{\nu} \|v\|_{L^q}\|\omega\|_{L^2}^{\frac{2}{p}} 
    \|\omega\|_{\dot{H}^1}^{\frac{3}{q}}
    \leq
    \frac{C_p}{\nu^p}\|v\|_{L^q}^p\|\omega\|_{L^2}^2
    +2\|\omega\|_{\dot{H}^1}^2.
\end{equation}
This immediately implies that
\begin{equation}
    -2 \nu \|\omega\|_{\dot{H}^1}^2+
    C \|v\|_{L^q}\|\omega\|_{L^2}^{\frac{2}{p}} 
    \|\omega\|_{\dot{H}^1}^{\frac{3}{q}}
    \leq
    \frac{C_p}{\nu^{p-1}}\|v\|_{L^q}^p \|\omega\|_{L^2}^2.
\end{equation}
Therefore we may conclude that
for all $0<t<T_{max}$
\begin{equation}
    \partial_t \|\omega(\cdot,t)\|_{L^2}^2
    \leq
    \left(\frac{C_p}{\nu^{p-1}}\|v\|_{L^q}^p
    +\sqrt{2} \|\sigma\|_{L^\infty}\right)
    \|\omega\|_{L^2}^2.
\end{equation}
Applying Gr\"onwall's inequality, we find that for all $0<T<T_{max},$
\begin{equation}
    \|\omega(\cdot,T)\|_{L^2}^2 \leq 
    \left\| \omega^0 \right\|_{L^2}^2
    \exp\left(\frac{C_p}{\nu^{p-1}} \int_0^T 
    \|v(\cdot,t)\|_{L^q}^p \diff t
    +\sqrt{2}\int_0^T \|\sigma(\cdot,t)\|_{L^\infty} \diff t
    \right).
\end{equation}
This completes the proof.
\end{proof}

We will now prove a corollary that requires the concentration of the $L^p_T L^q_x$ norm of $\omega$ at large values in the range 
for all 
$\frac{3}{2}<q<+\infty, \frac{2}{p}+\frac{3}{q}=2$,
as $t\to T_{max}$, when $T_{max}<+\infty.$

\begin{corollary} \label{VortCor}
Suppose $u\in C\left(\left[0,T_{max}\right);
\dot{H}^1_{df}\right)$ is a smooth solution of the Navier--Stokes equation, and
suppose $h\in L^1\left(\left[0,T_{max}\right);
\mathbb{R}^+\right)$. 
Let $\frac{3}{2}<q<+\infty,
\frac{2}{p}+\frac{3}{q}=2,$ and let
\begin{equation}
    v(x,t)= \begin{cases}
        \omega(x,t), &\text{if } |\omega(x,t)|> h(t)
        \\
        0, &\text{if } |\omega(x,t)|\leq h(t) 
        \end{cases}.
\end{equation}
Then for all $0<T<T_{max}$
\begin{equation} \label{EnstrophyBoundVortCor}
    \|\omega(\cdot,T)\|_{L^2}^2 \leq 
    \left\| \omega^0 \right\|_{L^2}^2
    \exp\left(\frac{C_p}{\nu^{p-1}} \int_0^T 
    \|v(\cdot,t)\|_{L^q}^p \diff t
    +\sqrt{2}\int_0^T  h(t) \diff t
    \right),
\end{equation}
where $C_p$ depends only on $p.$
In particular if $T_{max}<+\infty,$ then
\begin{equation}
    \int_0^{T_{max}} 
    \|v(\cdot,t)\|_{L^q}^p \diff t
    =+\infty.
\end{equation}
\end{corollary}

\begin{proof}
We will begin by letting
\begin{equation}
    \sigma(x,t)= \begin{cases}
        \omega(x,t), &\text{if } |\omega(x,t)|\leq h(t) \\
        0, &\text{if } |\omega(x,t)|> h(t)
        \end{cases}.
\end{equation}
We can see immediately that for all $0<t<T_{max},$
\begin{equation}
   \|\sigma(\cdot,t)\|_{L^\infty}\leq h(t), 
\end{equation}
and that
\begin{equation}
    \omega=v+\sigma.
\end{equation}
Therefore we can apply Theorem \ref{VortSumSpace} and find that
\begin{align}
    \|\omega(\cdot,T)\|_{L^2}^2 
    &\leq 
    \left\| \omega^0 \right\|_{L^2}^2
    \exp\left(\frac{C_p}{\nu^{p-1}} \int_0^T 
    \|v(\cdot,t)\|_{L^q}^p \diff t
    +\sqrt{2}\int_0^T \|\sigma(\cdot,t)\|_{L^\infty} \diff t
    \right)\\
    &\leq
    \left\| \omega^0 \right\|_{L^2}^2
    \exp\left(\frac{C_p}{\nu^{p-1}} \int_0^T 
    \|v(\cdot,t)\|_{L^q}^p \diff t
    +\sqrt{2}\int_0^T  h(t) \diff t
    \right).
\end{align}

Next we will note, as in Theorem \ref{VortSumSpace}, that if $T_{max}<+\infty,$ then
\begin{equation}
    \lim_{T\to T_{max}}\|\omega(\cdot,T)\|_{L^2}^2=+\infty.
\end{equation}
Therefore we can conclude that if $T_{max}<+\infty,$ then
\begin{equation}
    \frac{C_p}{\nu^{p-1}} \int_0^{T_{max}} 
    \|v(\cdot,t)\|_{L^q}^p \diff t
    +\sqrt{2}\int_0^{T_{max}}  h(t) \diff t
    =+\infty.
\end{equation}
However, we know by hypothesis that
\begin{equation}
    \int_0^{T_{max}}  h(t) \diff t<+\infty,
\end{equation}
so we may conclude that
\begin{equation}
    \int_0^{T_{max}} 
    \|v(\cdot,t)\|_{L^q}^p \diff t
    =+\infty.
\end{equation}
This completes the proof.
\end{proof}

Using Theorem \ref{VortSumSpace} and Theorem \ref{MixedLpSum}, we will now extend our regularity criterion for the vorticity from the space 
$L^p_T L^q_x+ L^1_T L^\infty_x$
to the slightly larger space
$L^p_T L^{q,\infty}_x+ L^1_T L^\infty_x$,
proving Corollary \ref{VortSumSpaceWeakIntro}, which is restated here for the reader's convenience.

\begin{corollary} \label{VortSumSpaceWeak}
Suppose $u\in C\left(\left[0,T_{max}\right);
\dot{H}^1_{df}\right)$ is a smooth solution of the Navier--Stokes equation. 
Let $\frac{3}{2}<q<+\infty, \frac{2}{p}+\frac{3}{q}=2,$ 
and let 
$\omega=v+\sigma$. 
Then for all $0<T<T_{max},$
\begin{equation} \label{EnstrophyBoundVortWeak}
    \|\omega(\cdot,T)\|_{L^2}^2 \leq 
    \left\| \omega^0 \right\|_{L^2}^2
    \exp\left(\Tilde{C_p} \int_0^T 
    \|v(\cdot,t)\|_{L^{q,\infty}}^p \diff t
    +\sqrt{2}\int_0^T \|\sigma(\cdot,t)\|_{L^\infty} 
    \diff t \right),
\end{equation}
where
\begin{equation}
    \Tilde{C}_p=
    \frac{C_{p'}}{\nu^{p'-1}}
    \left(\frac{q}{q-q'}\right)^\frac{p'}{q'}
    +\sqrt{2},
\end{equation}
with $C_{p'}$ taken as in Theorem \ref{VortSumSpace},
and $\frac{3}{2}<q'<q, \frac{2}{p'}+\frac{3}{q'}=2.$
In particular if $T_{max}<+\infty,$ then
\begin{equation}
    \Tilde{C_p} \int_0^T 
    \|v(\cdot,t)\|_{L^{q,\infty}}^p \diff t
    +\sqrt{2}\int_0^T \|\sigma(\cdot,t)\|_{L^\infty} 
    \diff t
    =+\infty.
\end{equation}
\end{corollary}

\begin{proof}
We know that if $T_{max}<+\infty,$ then 
\begin{equation}
    \lim_{T \to T_{max}}\|\omega(\cdot,T)\|_{L^2}^2
    =+\infty,   
\end{equation}
so it suffices to prove the bound
\eqref{EnstrophyBoundVortWeak}.

We will begin fixing $\frac{3}{2}<q'<q$ and by setting
\begin{equation}
    \phi(x,t)= \begin{cases}
        v(x,t), &\text{if } |v(x,t)|>
        \|v(\cdot,t)\|_{L^{q,\infty}}^p
        \\
        0, &\text{if } |v(x,t)|\leq
        \|v(\cdot,t)\|_{L^{q,\infty}}^p
            \end{cases},
\end{equation}
and
\begin{equation}
    \psi(x,t)= \begin{cases}
        v(x,t), &\text{if } |v(x,t)| \leq
        \|v(\cdot,t)\|_{L^{q,\infty}}^p
        \\
        0, &\text{if } |v(x,t)| >
        \|v(\cdot,t)\|_{L^{q,\infty}}^p
            \end{cases}.
\end{equation}
It is clear that $v=\phi+\psi$,
and applying Theorem \ref{MixedLpSum},
we have the bounds
\begin{equation} \label{VortBoundA}
    \int_0^T \|\phi(\cdot,t)\|_{L^{q'}}^{p'} \diff t
    \leq \left(\frac{q}{q-q'}\right)^\frac{p'}{q'} 
    \int_0^T \|v(\cdot,t)\|_{L^{q,\infty}}^p \diff t,
\end{equation}
and
\begin{equation} \label{VortBoundB}
    \int_0^T \|\psi(\cdot,t)\|_{L^\infty} \diff t
    \leq \int_0^T 
    \|v(\cdot,t)\|_{L^{q,\infty}}^p \diff t.
\end{equation}
Recalling that $\omega=\phi+\psi+\sigma,$ and applying \eqref{VortBoundA}, \eqref{VortBoundB}, and Theorem \ref{VortSumSpace}, we can conclude that
for all $\frac{3}{2}<q'<q,$
and for all $0<T<T_{max}$
\begin{align}
    \|\omega(\cdot,T)\|_{L^2}^2 
    &\leq 
    \left\| \omega^0 \right\|_{L^2}^2
    \exp\left(\frac{C_{p'}}{\nu^{p'-1}} \int_0^T 
    \|\phi(\cdot,t)\|_{L^{q'}}^{p'} \diff t
    +\sqrt{2}\int_0^T \|\psi(\cdot,t)
    +\sigma(\cdot,t)\|_{L^\infty} \diff t \right) \\
    &\leq
    \left\| \omega^0 \right\|_{L^2}^2
    \exp\left(\left(\frac{C_{p'}}{\nu^{p'-1}}
    \left(\frac{q}{q-q'}\right)^\frac{p'}{q'} 
    +\sqrt{2}\right)
    \int_0^T \|v(\cdot,t)\|_{L^{q,\infty}}^p \diff t
    +\sqrt{2}\int_0^T \|\sigma(\cdot,t)\|_{L^\infty} \diff t
    \right).
\end{align}
This completes the proof.
\end{proof}

Finally, we will prove the endpoint regularity criterion, Theorem \ref{VortEndpointIntro}, which is restated here for the reader's convenience.

\begin{theorem} \label{VortEndpoint}
Suppose $u\in C\left(\left[0,T_{max}\right);
\dot{H}^1_{df}\right)$ is a smooth solution of the Navier--Stokes equation, and
suppose $h\in L^1\left(\left[0,T_{max}\right);
\mathbb{R}^+\right)$. 
Let
\begin{equation}
    v(x,t)= \begin{cases}
        \omega(x,t), &\text{if } |\omega(x,t)|> h(t)
        \\
        0, &\text{if } |\omega(x,t)|\leq h(t)
        \end{cases}.
\end{equation}
If $T_{max}<+\infty,$ then
\begin{equation}
    \limsup_{t \to T_{max}}
    \left\|v(\cdot,t)\right\|_{L^\frac{3}{2}}
    \geq \frac{3 \pi^\frac{4}{3}}{2^\frac{5}{6}} \nu.
\end{equation}
\end{theorem}

\begin{proof}
Suppose towards contradiction that $T_{max}<+\infty$ and
\begin{equation}
    \limsup_{t \to T_{max}}
    \left\|v(\cdot,t)\right\|_{L^\frac{3}{2}}
    < \frac{3 \pi^\frac{4}{3}}{2^\frac{5}{6}} \nu.
\end{equation}
This means that there exists $\epsilon>0,$ such that for all $T_{max}-\epsilon<t<T_{max},$
\begin{equation}
    \left\|v(\cdot,t)\right\|_{L^\frac{3}{2}}
    < \frac{3 \pi^\frac{4}{3}}{2^\frac{5}{6}} \nu.
\end{equation}
We will again let
\begin{equation}
    \sigma(x,t)= \begin{cases}
        \omega(x,t), &\text{if } |\omega(x,t)|\leq h(t) \\
        0, &\text{if } |\omega(x,t)|> h(t)
        \end{cases}.
\end{equation}
We can see immediately that for all $0<t<T_{max},$
\begin{equation}
   \|\sigma(\cdot,t)\|_{L^\infty}\leq h(t), 
\end{equation}
and that
\begin{equation}
    \omega=v+\sigma.
\end{equation}
Applying our standard identity for enstrophy growth we can see that
\begin{align}
    \partial_t \frac{1}{2} \|\omega(\cdot,t)\|_{L^2}^2
    &=
    -\nu \|\omega\|_{\dot{H}^1}^2
    +\left<S;\omega \otimes \omega\right>\\
    &=
    -\nu \|\omega\|_{\dot{H}^1}^2
    +\left<S\omega; v\right>
    +\left<S\omega; \sigma \right>.
\end{align}
Applying H\"older's inequality with exponents $\frac{3}{2},3$ and $1,\infty,$ the Sobolev inequality, and the isometry in Proposition \ref{isometry}.
\begin{align}
    \partial_t \frac{1}{2} \|\omega(\cdot,t)\|_{L^2}^2
    &\leq 
    -\nu \|\omega\|_{\dot{H}^1}^2
    +\|S\omega\|_{L^3} \|v\|_{L^\frac{3}{2}}
    +\|S\omega\|_{L^1}\|\sigma\|_{L^\infty}\\
    &\leq 
    -\nu \|\omega\|_{\dot{H}^1}^2
    +\|S\|_{L^6} \|\omega\|_{L^6} \|v\|_{L^\frac{3}{2}}
    +h \|S\|_{L^2}\|\omega\|_{L^2} \\
    &\leq
    -\nu \|\omega\|_{\dot{H}^1}^2
    +\frac{2^\frac{4}{3}}{3 \pi^\frac{4}{3}}
    \|S\|_{\dot{H}^1} \|\omega\|_{\dot{H}^1}
    \|v\|_{L^\frac{3}{2}}
    +h \|S\|_{L^2}\|\omega\|_{L^2} \\
    &=
    -\nu \|\omega\|_{\dot{H}^1}^2
    + \frac{2^\frac{5}{6}}{3 \pi^\frac{4}{3}}
    \|\omega\|_{\dot{H}^1}^2 \|v\|_{L^\frac{3}{2}}
    +\frac{1}{\sqrt{2}} h\|\omega\|_{L^2}\\
    &=
    -\nu \|\omega\|_{\dot{H}^1}^2 \left(
    1-\frac{2^\frac{5}{6}}{3\pi^\frac{4}{3} \nu}
    \|v\|_{L^\frac{3}{2}} \right)
    +\frac{1}{\sqrt{2}} h\|\omega\|_{L^2}
\end{align}

Recall that by hypothesis 
for all $T_{max}-\epsilon<t<T_{max},$
\begin{equation}
    \frac{2^\frac{5}{6}}{3\pi^\frac{4}{3} \nu}
    \|v\|_{L^\frac{3}{2}} <1,
\end{equation}
Therefore we can conclude that for all
$T_{max}-\epsilon<t<T_{max},$
\begin{equation}
    \partial_t \|\omega\|_{L^2}^2
    \leq \sqrt{2} h \|\omega\|_{L^2}^2.
\end{equation}
Applying Gr\"onwall's inequality, this implies that
for all $T_{max}-\epsilon<T<T_{max},$
\begin{equation}
    \|\omega(\cdot,T)\|_{L^2}^2 \leq
    \left\|\omega(\cdot,T_{max} -\epsilon)\right\|_{L^2}^2
    \exp \left( \sqrt{2} \int_{T_{max}-\epsilon}^T
    h(t) \diff t\right).
\end{equation}
Using the assumption that
$h\in L^1\left(\left[0,T_{max}\right);
\mathbb{R}^+\right),$ we can conclude that
\begin{align}
    \limsup_{T \to T_{max}} \|\omega(\cdot,T)\|_{L^2}^2
    &\leq
    \left\|\omega(\cdot,T_{max}-\epsilon) \right\|_{L^2}^2
    \exp\left(\sqrt{2} \int_{T_{max}-\epsilon}^{T_{max}} 
    h(t) \diff t\right)\\
    &<
    +\infty.
\end{align}
This contradicts our assumption that $T_{max}<+\infty,$ so this completes the proof.
\end{proof}

\bibliographystyle{plain}
\bibliography{Bib}
\end{document}